\documentclass[10pt]{article}
\usepackage{amssymb,bbm,amsthm}
\usepackage{amsmath}
\usepackage{amsbsy}
\usepackage{color}
\usepackage{cite}
\usepackage[dvips]{graphicx}
\usepackage{wasysym}

\newcommand{\as}{\\[.6em]}

\newcommand{\AS}{\\[1.2em]}
\newcommand{\dis}{\displaystyle}
\newcommand{\bela}[1]{\begin{equation}\label{#1}}
\newcommand{\ela}{\end{equation}}
\newcommand{\bear}[1]{\begin{array}{#1}}
\newcommand{\ear}{\end{array}}

\newcommand{\diag}{\operatorname{diag}}
\newcommand{\one}{\mathbbm{1}}

\newcommand{\tb}{\bar{2}}
\newcommand{\ob}{\bar{1}}
\newcommand{\hb}{\bar{3}}
\newcommand{\Z}{\mathbbm{Z}}
\renewcommand{\P}{\mathbbm{P}}
\newcommand{\R}{\mathbbm{R}}
\newcommand{\C}{\mathbbm{C}}
\newcommand{\De}{D^{\rm e}}
\newcommand{\Do}{D^{\rm o}}

\newcommand{\Msf}{\mathsf{M}}
\newcommand{\Vsf}{\mathsf{V}}

\newcommand{\lsf}{\mathsf{l}}
\newcommand{\psf}{\mathsf{p}}
\newcommand{\asf}{\mathsf{a}}
\newcommand{\bsf}{\mathsf{b}}
\newcommand{\csf}{\mathsf{c}}
\newcommand{\gsf}{\mathsf{g}}
\newcommand{\Lsf}{\mathsf{L}}
\newcommand{\ssf}{\mathsf{s}}
\newcommand{\xsf}{\mathsf{x}}
\newcommand{\ysf}{\mathsf{y}}
\newcommand{\Bsf}{\mathsf{B}}

\newcommand{\Csf}{\mathsf{C}}
\newcommand{\Esf}{\mathsf{E}}
\newcommand{\Lambdasf}{\mathsf{\Lambda}}

\newcommand{\calM}{\mathcal{M}}
\newcommand{\calV}{\mathcal{V}}
\newcommand{\calW}{\mathcal{W}}

\newcommand{\dualu}{\mathfrak{u}}
\newcommand{\dualv}{\mathfrak{v}}

\theoremstyle{plain}
\newtheorem{theorem}{Theorem}[section]
\newtheorem{corollary}[theorem]{Corollary}
\newtheorem{lemma}[theorem]{Lemma}

\theoremstyle{definition}
\newtheorem{definition}[theorem]{Definition}
\newtheorem{remark}[theorem]{Remark}

\numberwithin{equation}{section}

\begin{document}

\begin{center} 
 {\Large\bf Circle complexes \medskip and the\\ discrete CKP equation}
\end{center}

\smallskip

\begin{center}
 \sc
 A.I.\ Bobenko$\,^{1}$ and W.K.\ Schief$\,^{2,3}$ 
\end{center}

\smallskip

\begin{center}
 \small\sl
$^1$ Institut f\"ur Mathematik, Technische Universit\"at Berlin, Stra\ss e des 17.\ Juni 136, 10623 Berlin, Germany\\[2mm]
$^2$ School of Mathematics and Statistics, The University of New South Wales, Sydney, NSW 2052, Australia\\[2mm]
$^3$ Australian Research Council Centre of Excellence for Mathematics and Statistics of Complex Systems, School of Mathematics and Statistics, The University of New South Wales, Sydney, NSW 2052, Australia
\end{center}

\begin{abstract}
In the spirit of Klein's Erlangen Program, we investigate the geometric and algebraic structure of fundamental line complexes and the underlying privileged discrete integrable system for the minors of a matrix which constitute associated Pl\"ucker coordinates. Particular emphasis is put on the restriction to Lie circle geometry which is intimately related to the master dCKP equation of discrete integrable systems theory. The geometric interpretation, construction and integrability of fundamental line complexes in M\"obius, Laguerre and hyperbolic geometry are discussed in detail. In the process, we encounter various avatars of classical and novel incidence theorems and associated cross- and multi-ratio identities for particular hypercomplex numbers. This leads to a discrete integrable equation which, in the context of M\"obius geometry, governs novel doubly hexagonal circle patterns. 
\end{abstract}

\section{Introduction}

Line congruences, that is, two-parameter families of lines, constitute fundamental objects in classical differential geometry \cite{Finikov59}. In particular, normal and Weingarten congruences have been studied in great detail \cite{Eisenhart60}. Their importance in connection with the geometric theory of integrable systems has been well documented (see \cite{Doliwa01} and references therein). Recently, in the context of integrable discrete differential geometry \cite{BobenkoSuris09}, attention has been drawn to {\em discrete line congruences}~\cite{DoliwaSantiniManas00}, that is, two-parameter families of lines which are (combinatorially) attached to the vertices of a $\Z^2$ lattice. Discrete Weingarten congruences have been shown to lie at the heart of the B\"acklund transformation for discrete pseudospherical surfaces \cite{Sauer50,Wunderlich51,BobenkoPinkall96}. Discrete normal congruences have been used to define Gaussian and mean curvatures and the associated Steiner formula for discrete analogues of surfaces parametrised in terms of curvature coordinates \cite{Schief03,Schief06,BobenkoPottmannWallner2010}. Discrete line congruences have also found important applications in architectural geometry \cite{WangJiangBompasPottmann13}.

Discrete line congruences in a projective space of arbitrary dimension may be consistently extended to ``lattices of lines'' of $\Z^N$ combinatorics if these admit some elementary geometric properties. These lattices of lines are termed rectilinear line congruences in \cite{DoliwaSantiniManas00}. Their essential properties are encoded in the three-dimensional sublattices of lines, that is, maps of the form
\bela{I1}
  \lsf : \Z^3\rightarrow{\{\mbox{lines in $\C\P^d$}}\}.
\ela
A combinatorial picture of an elementary cube of such a lattice is depicted in Figure \ref{linecomplex}.
\begin{figure}
\centerline{\includegraphics[scale=0.5]{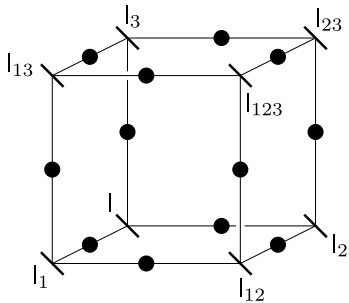}}
\caption{A combinatorial picture of a line complex $\lsf$ which admits the ``intersection property''. The lines are combinatorially attached to the vertices of a $\Z^3$ lattice. Any two neighbouring lines intersect in a point which may be associated with the corresponding edge.}
\label{linecomplex}
\end{figure} 
It will be seen that the assumption of a complex projective space is crucial in the classification of these three-parameter families of lines which we have termed (discrete) {\em fundamental line complexes} in \cite{BobenkoSchief15}. For $d\geq4$, these are characterised by the property that neighbouring lines intersect (cf.\ Figure~\ref{linecomplex}). In the case of a three-dimensional ambient space ($d=3$), one imposes the additional requirement that the discrete line complexes may be regarded as projections of discrete line complexes (which have the ``intersection property'') in a four-dimensional projective space. In \cite{BobenkoSchief15}, it has been demonstrated that this property is naturally encoded in a geometric {\em correlation} between the two lines attached to any pair of ``opposite'' vertices of any elementary cube of the $\Z^3$ lattice such as the pair $\lsf,\lsf_{123}$ in Figure \ref{linecomplex}. Thus, a ``hexagon of lines'' such as $[\lsf_1,\lsf_{12},\lsf_2,\lsf_{23},\lsf_3,\lsf_{13}]$ uniquely determines a correlation of  $\C\P^3$ \cite{OnishchikSulanke06,SempleKneebone52}
which maps opposite lines to each other. Given a member $\lsf$ of the one-parameter family of lines intersecting the lines $\lsf_1,\lsf_2,\lsf_3$, this polarity selects a line $\lsf_{123}$ which intersects the lines $\lsf_{12},\lsf_{23},\lsf_{13}$. It turns out that the notion of correlated lines plays a key role in the analysis of canonical classes of fundamental line complexes.

In \cite{BobenkoSchief15}, fundamental line complexes in $\C\P^3$ have been shown to be governed by a privileged discrete system which is multi-dimensionally consistent \cite{BobenkoSuris09}. The algebraic integrability of this {\em $M$-system} implies the geometric integrability of fundamental line complexes in $\C\P^3$. In concrete terms, the $M$-system may be regarded as an integrable system of equations for the minors of a matrix $\calM$. This integrable system determines the discrete evolution of a $2\times2$ submatrix, the minors of which constitute the Pl\"ucker coordinates of the lines of the fundamental line complexes. Any admissible constraint imposed on the $M$-system restricts the class of admissible lines and therefore the admissible points in the associated four-dimensional complex Pl\"ucker quadric. This is the starting point of our discussion in this paper. Thus, if the matrix $\calM$ is real then we obtain real fundamental line complexes. If $\mathcal{M}$ is Hermitian then the restriction of the complex Pl\"ucker quadric may be identified with a real Lie quadric of signature $(4,2)$ \cite{Blaschke29} so that the fundamental complexes of intersecting lines may be interpreted as fundamental complexes of touching oriented spheres in $\R^3$. If, in addition, $\cal{M}$ is real so that $\calM$ is symmetric then the Pl\"ucker quadric is restricted to a three-dimensional real quadric of signature $(3,2)$ which may be interpreted as a Lie quadric. Accordingly, the lines of the fundamental complex may be interpreted as circles in the real plane which touch each other in an oriented manner \cite{Blaschke29}. This is the case which we investigate in great detail in this paper. In algebraic terms, it is encoded in the master dCKP equation of discrete integrable systems theory \cite{Kashaev96,Schief03b} which has also appeared in various other contexts such as cluster algebras and dimer configurations \cite{KenyonPemantle13} and the `principal minor assignment problem'~\cite{HoltzSturmfels07}. It turns out that the action of the above-mentioned correlation of $\C\P^3$ leaves invariant the restriction of the Pl\"ucker quadric to the Lie quadric. Accordingly, the correlation acts in such a manner that circles are mapped to circles. The multi-dimensional consistency of the real symmetric $M$-system therefore implies that planar lattices of touching correlated circles of $\Z^3$ combinatorics are integrable in that these fundamental circle complexes constitute multi-dimensionally consistent 3D geometric configurations. Hence, two-dimensional sublattices of fundamental circle complexes, that is, oriented circles packings of $\Z^2$ combinatorics are integrable in the same sense as discrete asymptotic, conjugate and curvature nets \cite{BobenkoSuris09}. 

The construction of fundamental circle complexes is intimately related to the classical Apollonius problem \cite{Coolidge16}. Indeed, in the case of fundamental circle complexes, the lines in Figure \ref{linecomplex} are interpreted as circles. Hence, given a ``hexagon of circles" $[\csf_1,\csf_{12},\csf_2,\csf_{23},\csf_3,\csf_{13}]$ which are in oriented contact, there exist at most two circles $\csf$ which are in oriented contact with the circles $\csf_1,\csf_2,\csf_3$ and at most two oriented circles $\csf_{123}$ which touch the circles $\csf_{12},\csf_{23},\csf_{13}$. The correlation associated with the hexagon of circles then provides a link between the solutions of these two Apollonius problems. Thus, any choice of the circle $\csf$ leads to a unique choice of the circle $\csf_{123}$. Since this correlation of circles renders fundamental circle complexes integrable, it is important to examine its algebraic and geometric properties in more detail. In terms of line geometry, restricting the real Pl\"ucker quadric to a Lie quadric corresponds to confining the class of admissible lines in $\R\P^3$ to a {\em linear line complex} \cite{SempleKneebone52}. Thus, there exists a one-to-one correspondence between oriented circles in the plane and lines of a linear line complex. As mentioned in the preceding, the correlation acts within the linear line complex and, hence, the correlation of two lines of a linear line complex corresponds to the correlation of two circles. We show that there exists an elementary M\"obius geometric construction of pairs of correlated circles which is also meaningful in Lie geometric terms. 

It is well known that Lie circle geometry contains M\"obius, Laguerre and hyperbolic geometry (in a sense to be specified) as subgeometries. It is therefore natural to inquire as to the existence of fundamental complexes in these geometries. To this end, we first examine in detail fundamental point-circle complexes in M\"obius geometry. Their existence in terms of a well-posed Cauchy problem and their integrability are shown both geometrically and algebraically. For the construction of these complexes and their multi-dimensional extensions, we make use of known and novel incidence theorems which may be proven in a purely combinatorial manner by means of cross-ratio identities. These include Miquel's and Clifford's circle theorems \cite{Clifford71,Pedoe88}. If the complex numbers in the cross-ratios are replaced by double numbers \cite{Yaglom68} then the analysis undertaken in the M\"obius geometric context may be translated directly into the language of hyperbolic geometry. This leads to the construction and integrability of fundamental geodesic-circle complexes. Furthermore, the same approach is valid if one considers dual numbers \cite{Yaglom68}, leading to line-circle complexes in Laguerre geometry. However, interestingly, these turn out to be anti-fundamental in a sense to be discussed subsequently. The (anti-)fundamental complexes in the subgeometries of M\"obius geometry considered here are shown to be governed by a novel 2D discrete integrable equation involving multi-ratios. This is consistent with the fact that the geometric Cauchy data for the complexes are one-dimensional. Indeed, the construction of an $A_2$ ``slice'' of a(n) (anti-)fundamental complex suffices to determine the entire complex. In the context of M\"obius geometry, this implies that fundamental point-circle complexes are encoded in novel doubly hexagonal circle patterns. In this connection, it noted that other types of integrable hexagonal circle patterns have been discussed in \cite{BobenkoHoffmannSuris02,BobenkoHoffmann03}.

\section{\boldmath Discrete line complexes and the $M$-system}

Here, we review the relevant geometric and algebraic properties of the fundamental line complexes analysed in detail in \cite{BobenkoSchief15}. A {\em (discrete) line complex} in a three-dimensional complex projective space $\C\P^3$ is a three-parameter family of lines which are combinatorially attached to the vertices of a $\Z^3$ lattice, that is, a map
\bela{E1}
  \lsf:\Z^3\rightarrow\{\mbox{lines in $\C\P^3$}\}.
\ela
Whenever possible, we suppress the discrete independent variables and indicate their relative increments by, for instance,
\bela{E2}
  \lsf = \lsf(n_1,n_2,n_3),\quad \lsf_1 = \lsf(n_1+1,n_2,n_3),\quad \lsf_{23} = \lsf(n_1,n_2+1,n_3+1).
\ela
Decrements are encoded in the notation $\lsf_{\bar{1}} = \lsf(n_1-1,n_2,n_3)$. If two neighbouring lines $\lsf$ and $\lsf_l$ intersect then we denote their point of intersection by
\bela{E3}
  [\psf^l] = \lsf\cap\lsf_l
\ela
with $\psf^l\in\C^4$ being homogeneous coordinates of $[\psf^l]\in\C\P^3$. However, in the following, as far as  nomenclature is concerned, we will not distinguish between objects in a projective space and their representation in terms of homogeneous coordinates. For instance, we will simply refer to $\psf^l$ as the point of intersection of $\lsf$ and $\lsf_l$.

\subsection{Fundamental line complexes}

{\em Fundamental line complexes} are now defined in terms of the eight lines attached to the vertices of any elementary cube of the $\Z^3$ lattice (cf.\ Figure \ref{fundamental}).
\begin{figure}
\centerline{\includegraphics[scale=0.5]{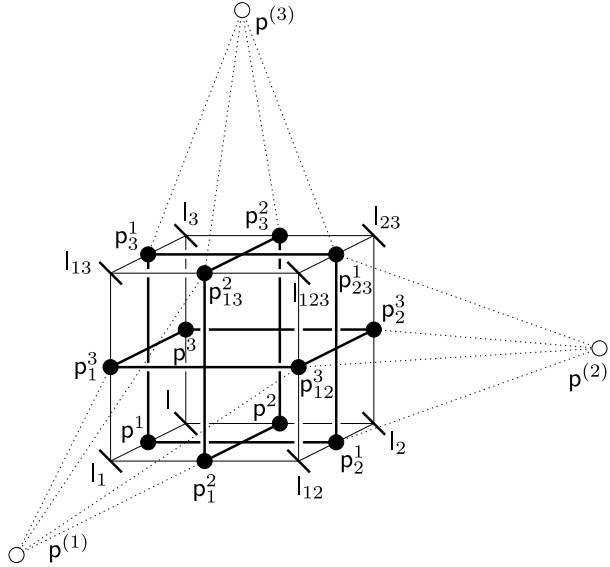}}
\caption{A combinatorial picture of a fundamental line complex $\lsf$. The four points of intersection attached to any four edges of the same ``type'' are coplanar. This is equivalent to the concurrency of any four ``diagonals'' of the same ``type'' indicated by the dotted lines.}
\label{fundamental}
\end{figure} 
\begin{definition}\label{line_complex}
A line complex $\lsf:\Z^3\rightarrow\{\mbox{lines in $\C\P^3$}\}$ is termed {\em fundamental} if any neighbouring lines $\lsf$ and $\lsf_l$ intersect and the points of intersection $\psf^l$ enjoy the {\em coplanarity property}, that is, for any distinct $l,m$ and $p$, the quadrilaterals $[\psf^l,\psf^l_m,\psf^l_{mp},\psf^l_p]$ are planar. Equivalently, the lines passing through the points $\psf^l$ and $\psf^l_m$ admit the {\em concurrency property}, that is, the four lines connecting the points $\psf^l,\psf^p,\psf^l_p,\psf^p_l$ and their shifted counterparts $\psf^l_m,\psf^p_m,\psf^l_{pm},\psf^p_{lm}$ are concurrent.
\end{definition}

\subsection{The Pl\"ucker quadric}

The classical Pl\"ucker correspondence \cite{OnishchikSulanke06,Plucker65} may be used to identify a line $\lsf$ in $\C\P^3$ with a point in the four-dimensional Pl\"ucker quadric $Q^4$ embedded in a five-dimensional complex projective space $\C\P^5$. Homogeneous coordinates 
\bela{E4}
  \Vsf = (\gamma^{01},\gamma^{23},\gamma^{02},\gamma^{13},\gamma^{03},\gamma^{12})\in\C^6
\ela
of this point are constructed by considering homogeneous coordinates
\bela{E5}
  \asf = \left(\bear{c}\alpha^0\\ \alpha^1\\ \alpha^2\\ \alpha^3\ear\right),\quad 
  \bsf = \left(\bear{c}\beta^0\\ \beta^1\\ \beta^2\\ \beta^3\ear\right)
\ela
of two distinct points $\asf$ and $\bsf$ on the line $\lsf$ and defining the Pl\"ucker coordinates of the line by
\bela{E6}
  \gamma^{\mu\nu} = \det\left(\bear{cc} \alpha^\mu & \beta^\mu\\ \alpha^\nu & \beta^\nu\ear\right).
\ela
The Pl\"ucker quadric $Q^4$ is encapsulated in the Pl\"ucker identity
\bela{E7}
  \gamma^{01}\gamma^{23} - \gamma^{02}\gamma^{13} + \gamma^{03}\gamma^{12} = 0.
\ela
It is convenient to formulate the latter as
\bela{E8}
  \langle \Vsf,\Vsf\rangle = 0,
\ela
where the inner product is taken with respect to the block-diagonal metric
\bela{E9}
  \diag\left[\left(\bear{cc}0&1\\1&0\ear\right),-\left(\bear{cc}0&1\\1&0\ear\right),\left(\bear{cc}0&1\\1&0\ear\right)\right].
\ela
Moreover, the condition of intersecting neighbouring lines $\lsf$ and $\lsf_l$ translates into
\bela{E10}
  \langle \Vsf,\Vsf_l\rangle = 0.
\ela

\subsection{The fundamental \boldmath $M$-system}

In \cite{BobenkoSchief15}, it has been demonstrated that fundamental line complexes are algebraically governed by a particular case of a fundamental discrete integrable system. In the current context, this {\em $M$-system} determines the ``evolution'' of a matrix 
\bela{E11}
  \calM = \left(\bear{ccccc} M^{11}&M^{12}&M^{13}&M^{14}&M^{15}\\
                            M^{21}&M^{22}&M^{23}&M^{24}&M^{25}\\
                            M^{31}&M^{32}&M^{33}&M^{34}&M^{35}\\
                            M^{41}&M^{42}&M^{43}&M^{44}&M^{45}\\
                            M^{51}&M^{52}&M^{53}&M^{54}&M^{55}
  \ear\right)
\ela
according to
\bela{E12}
  M^{ik}_l = M^{ik} - \frac{M^{il}M^{lk}}{M^{ll}},\quad l\in \{1,2,3\}\backslash\{i,k\}.
\ela
The submatrix 
\bela{E13}
  \hat{\calM} = \left(\bear{cc} M^{44}&M^{45}\\ M^{54}&M^{55}\ear\right)
\ela
parametrises the (generic) lines of the fundamental line complex via
\bela{E14}
  \Vsf = \left(1,\left|\bear{cc}M^{44}&M^{45}\\ M^{54}&M^{55}\ear\right|,M^{44},M^{55},M^{54},M^{45}\right).
\ela
It is easy to verify that $\langle \Vsf,\Vsf\rangle = 0$ identically and $\langle \Vsf,\Vsf_l\rangle=0$ is a consequence of the $M$-system (\ref{E12}). Furthermore, the coplanarity property may also be established \cite{BobenkoSchief15}. Hence, the minors of the matrix $\hat{\calM}$ constitute Pl\"ucker coordinates of the lines of a fundamental line complex. It is important to emphasise that all (generic) fundamental line complexes may be generated in this manner.

\subsection{Correlations}

In \cite{BobenkoSchief15}, it has been proven that fundamental line complexes may be characterised in terms of {\em correlations}. A correlation of a three-dimensional projective space is an incidence-preserving transformation which maps $k$-dimensional projective subspaces to $2-k$-dimensional projective subspaces \cite{OnishchikSulanke06,SempleKneebone52}. Thus, points, lines and planes are mapped to planes, lines and points respectively in such a manner that, for instance, the points of a line are mapped to planes which meet in a line. In terms of homogeneous coordinates, any correlation $\kappa$ is encoded in a matrix $\Bsf$ such that the image of a point $\xsf\in\C^4$ is given by 
\bela{E15}
  \kappa(\xsf) = \{\ysf\in\C^4 : \ysf^T\Bsf\xsf = 0\}.
\ela
It is known (see, e.g., \cite{Carver05}) that for any hexagon in $\C\P^3$ in general position there exists a unique correlation $\kappa$ which maps any edge (regarded as a line) to its ``opposite'' counterpart. The correlation $\kappa$ constitutes a polarity associated with a quadric defined by a symmetric matrix $\Bsf=\Bsf^T$. Specifically, if the vertices of the hexagon are denoted by $\xsf^i$ as indicated in Figure \ref{hexagon} and the planes spanned by any three successive vertices $\xsf^{i-1},\xsf^i,\xsf^{i+1}$ are labelled by $\pi^i$ then, in terms of homogeneous coordinates, there exists a unique symmetric matrix $\Bsf$ such that 
\bela{E16a}
  \kappa(\xsf^i) = \pi^{i+3},\quad i=1,\ldots,6
\ela
with indices taken modulo 6.
\begin{figure}
\centerline{\includegraphics[scale=0.5]{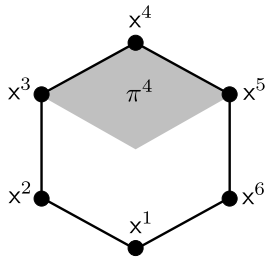}}
\caption{A hexagon in $\C\P^3$ with vertices $\xsf^i$ and associated planes $\pi^i$.}
\label{hexagon}
\end{figure}

We now consider an elementary cube of a line complex and assume that the lines $\lsf^{ii+1}$ as depicted in Figure \ref{correlation} form a hexagon, where $\lsf^{ii+1}$ passes through the vertices $\xsf^i$ and $\xsf^{ii+1}$.
\begin{figure}
\centerline{\includegraphics[scale=0.5]{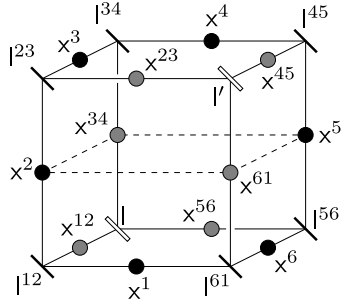}}
\caption{The hexagon with vertices $\xsf^i$ uniquely determines a correlation $\kappa$ which interchanges opposite lines $\lsf^{ii+1}$ and $\lsf^{i+3i+4}$. The correlation $\kappa$ then maps the given line $\lsf$ to a line $\lsf'$. The vertices $\xsf^i$ and the points of intersection $\xsf^{ii+1}$ enjoy the planarity property, that is, for instance, the points $\xsf^2,\xsf^{34},\xsf^5$ and $\xsf^{61}$ are coplanar.}
\label{correlation}
\end{figure}
There exists a one-parameter family of lines $\lsf$ which intersect the lines $\lsf^{12}$, $\lsf^{34}$ and $\lsf^{56}$. The correlation $\kappa$ maps any fixed line $\lsf$ to a unique line $\lsf'$ which intersects the lines $\lsf^{23}$, $\lsf^{45}$ and $\lsf^{61}$. Remarkably, the eight lines $\lsf,\lsf^{12},\lsf^{23},\lsf^{34},\lsf^{45},\lsf^{56},\lsf^{61}$ and $\lsf'$ turn out to form an elementary cube of a fundamental line complex, that is, for instance, the points of intersection $\xsf^2,\xsf^{34},\xsf^5$ and $\xsf^{61}$ may be shown to be coplanar \cite{BobenkoSchief15}. Since any elementary cube of a fundamental line complex is determined by seven lines via the planarity property, the existence of a correlation which maps the eight lines of an elementary cube of a fundamental line complex to their opposite counterparts is an equivalent characterisation of fundamental line complexes. It is observed that, in general, each elementary cube comes with a different correlation.

\section{Canonical reductions of the \boldmath $M$-system. Associated subgeometries}

In this section, we identify subgeometries of complex Pl\"ucker line geometry in which the definition of fundamental line complexes remains meaningful by considering admissible reductions of the $M$-system (\ref{E12}). 

\subsection{Real Pl\"ucker line geometry}

It is evident that the reality constraint
\bela{E17a}
  \bar{\mathcal{M}} = \mathcal{M}
\ela
is compatible with the $M$-system (\ref{E12}). Hence, the latter gives rise to $\Z^3$ lattices of points in a {\em real} Pl\"ucker quadric $Q^4$ with homogeneous coordinates
\bela{E16}
  \Vsf:\Z^3\rightarrow \R^{3,3},
\ela
where the space of homogeneous coordinates of signature $(3,3)$ is equipped with the metric (\ref{E9}). This corresponds, in turn, to fundamental line complexes
\bela{E17}
  \lsf : \Z^3 \rightarrow \{\mbox{lines in $\R\P^3$}\}
\ela
in real Pl\"ucker line geometry. In this case, the analysis undertaken in \cite{BobenkoSchief15} is still valid but it is seen below that the consideration of complex line complexes is essential in the classification of fundamental line complexes. 

\subsection{Lie circle geometry}

The $M$-system is invariant under exchanging left and right upper indices and it is therefore admissible to demand that the matrix $\mathcal{M}$ be symmetric, that is,
\bela{E18}
  \mathcal{M}^T = \mathcal{M}.
\ela
In particular, this implies that $M^{54} = M^{45}$ so that the admissible points in the Pl\"ucker quadric are contained in the hyperplane $\gamma^{03} = \gamma^{12}$. The intersection of the Pl\"ucker quadric with this hyperplane results in a three-dimensional quadric $Q^3$ so that the corresponding lines in $\C\P^3$ are no longer in general position but form a three-parameter family of lines $\mathcal{L}$ termed {\em linear line complex} due to the linear relation between the Pl\"ucker coordinates $\gamma^{ik}$ \cite{SempleKneebone52}. Accordingly, the symmetry constraint (\ref{E18}) results in fundamental line complexes 
\bela{E19}
  \lsf : \Z^3 \rightarrow \mathcal{L},
\ela
that is, fundamental line complexes which are such that each line belongs to the linear line complex $\mathcal{L}$.

In terms of homogeneous coordinates, the points in the quadric $Q^3$ are parametrised by
\bela{E20}
  \Vsf = \left(1,\left|\bear{cc}M^{44}&M^{45}\\ M^{45}&M^{55}\ear\right|,M^{44},M^{55},M^{45},M^{45}\right)
\ela
so that it is natural to introduce the reduced map
\bela{E21}
   \hat{\Vsf} :  \Z^3 \rightarrow \C^5
\ela
defined by 
\bela{E20a}
  \hat{\Vsf} = \left(1,\left|\bear{cc}M^{44}&M^{45}\\ M^{45}&M^{55}\ear\right|,M^{44},M^{55},M^{45}\right),
\ela
where the induced metric of the space of homogeneous coordinates $\C^5$ is given by
\bela{E21a}
  \diag\left[\left(\bear{cc}0&1\\1&0\ear\right),-\left(\bear{cc}0&1\\1&0\ear\right),2\right].
\ela
Thus, the quadric $Q^3$ is represented by $\langle\hat{\Vsf},\hat{\Vsf}\rangle=0$ and the fact that neighbouring lines $\lsf$ and $\lsf_l$ of a fundamental line complex intersect translates into $\langle\hat{\Vsf}_l,\hat{\Vsf}\rangle$=0. Accordingly, if the Pl\"ucker quadric is real then the quadric $Q^3$ has signature $(3,2)$ and may be identified with a Lie quadric, the points of which represent oriented circles $\csf$ in the real plane $\R^2$ \cite{Blaschke29}. Intersecting lines $\lsf$ and $\lsf_l$ of the linear line complex correspond to circles $\csf$ and $\csf_l$ which have oriented contact. Hence, the symmetric real $M$-system (\ref{E12}) gives rise to {\em fundamental circle complexes}
\bela{E22a}
  \csf : \Z^3 \rightarrow \{\mbox{oriented circles in $\R^2$}\}.
\ela
By construction, the two circles combinatorially attached to any neighbouring vertices of the $\Z^3$ lattice are necessarily in oriented contact. In Section 4, we determine the geometric implication of the planarity property which completely characterises fundamental circle complexes.

\subsection{Lie sphere geometry}

Fundamental complexes in Lie sphere geometry are obtained by considering Hermitian matrices
\bela{E23a}
  \mathcal{M}^\dagger = \mathcal{M}.
\ela
Indeed, the latter admissible constraint implies that $M^{54}=\bar{M}^{45}$ so that the relevant points in the Pl\"ucker quadric are parametrised by
\bela{E24a}
  \Vsf = \left(1,\left|\bear{cc}M^{44}&M^{45}\\ \bar{M}^{45}&M^{55}\ear\right|,M^{44},M^{55},\bar{M}^{45},M^{45}\right).
\ela
Real homogeneous coordinates are obtained by defining the map
\bela{E25a}
  \hat{\Vsf} : \Z^3 \rightarrow \R^{4,2},
\ela
where
\bela{E26a}
  \hat{\Vsf} = \left(1,\left|\bear{cc}M^{44}&M^{45}\\ \bar{M}^{45}&M^{55}\ear\right|,M^{44},M^{55},\Re( M^{45}),\Im(M^{45})\right)
\ela
and the induced metric of the space of homogeneous coordinates is given by
\bela{E21b}
  \diag\left[\left(\bear{cc}0&1\\1&0\ear\right),-\left(\bear{cc}0&1\\1&0\ear\right),2,2\right].
\ela
The property $\langle\hat{\Vsf},\hat{\Vsf}\rangle = 0$ is equivalent to $\langle\Vsf,\Vsf\rangle=0$ and encodes a real quadric $\hat{Q}^4$ of signature $(4,2)$. Hence, the lines $\lsf$ corresponding to the points (\ref{E24a}) in the Pl\"ucker quadric may be interpreted as oriented spheres $\ssf$ in a three-dimensional real Euclidean space $\R^3$ \cite{Blaschke29}. Once again, intersecting neighbouring lines $\lsf$ and $\lsf_l$ are characterised by $\langle\hat{\Vsf}_l,\hat{\Vsf}\rangle = 0$ which, in turn, means that the associated spheres $\ssf$ and $\ssf_l$ have oriented contact. Accordingly, the Hermitian $M$-system (\ref{E12}) gives rise to {\em fundamental sphere complexes}
\bela{E22}
  \ssf : \Z^3 \rightarrow \{\mbox{oriented spheres in $\R^3$}\}.
\ela
By construction, the two spheres combinatorially attached to any neighbouring vertices of the $\Z^3$ lattice are necessarily in oriented contact. The precise geometric implication of the planarity property which completely characterises fundamental sphere complexes is consigned to a forthcoming publication. Here, we confine ourselves to the analysis of the geometric and algebraic properties of the fundamental complexes associated with the symmetric $M$-system and its geometric reductions in M\"obius, Laguerre and hyperbolic geometry. 

The content of this section is summarised in Figure \ref{summary}.
\begin{figure}
\centerline{\includegraphics[width=0.7\textwidth]{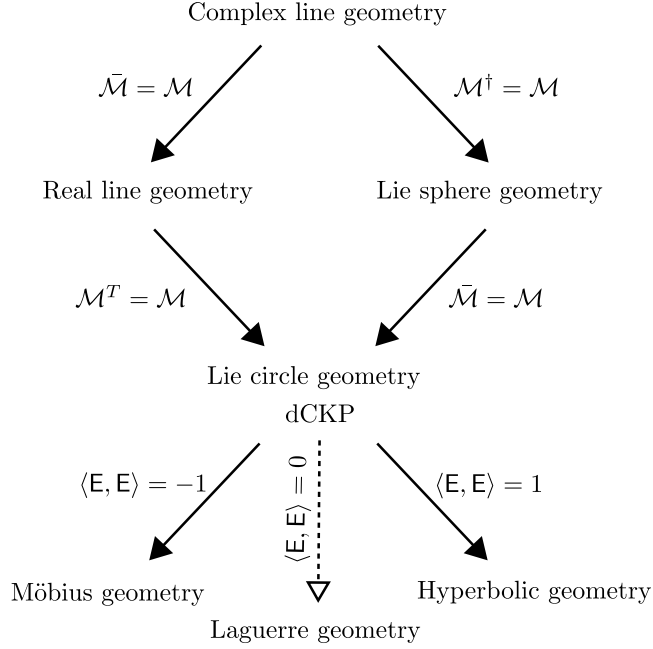}}
\caption{Subgeometries of complex line geometry. Fundamental complexes in real line, Lie sphere and Lie circle geometries are constructed by means of canonical reductions of the $M$-system. If $\calM$ is real and symmetric then the fundamental circle complexes are related to the integrable dCKP equation. Subgeometries of Lie circle geometry are obtained by intersecting the Lie quadric $Q^3$ with the polar plane of a point $\mathsf{E}$ with respect to $Q^3$. The cases $\langle\mathsf{E},\mathsf{E}\rangle=\pm1$ lead to fundamental complexes in M\"obius and hyperbolic geometries governed by particular reductions of the dCKP equation and its B\"acklund transforms. Anti-fundamental complexes in Laguerre geometry are associated with $\langle\mathsf{E},\mathsf{E}\rangle=0$.}
\label{summary}
\end{figure}
It is noted that the diagram closes in the sense that fundamental circle complexes are retrieved from fundamental sphere complexes by imposing the condition that $\mathcal{M}$ be real.

\section{Fundamental circle complexes}

In the preceding, it has been demonstrated that the symmetric $M$-system gives rise to fundamental line complexes which are constructed from lines belonging to the linear line complex $\mathcal{L}$ defined by $\gamma^{03}=\gamma^{12}$. If the lines are real then one may employ the classical equivalence of lines in a linear line complex (such as $\mathcal{L}$) and oriented circles so that any point in the Lie quadric $Q^3$ represented by $\Vsf$ may be interpreted as either a circle $\csf$ in the plane or a line $\lsf$ in a linear line complex. The connection between points in the Lie quadric and circles in the plane is made explicit in Section 6. Accordingly, in the current context, we may interpret fundamental line complexes as oriented circle complexes and it is therefore natural to adopt the following definition.

\begin{definition}
  A circle complex $\csf : \Z^3 \rightarrow \{\mbox{oriented circles in $\R^2$}\}$ is termed {\em fundamental} if the circles interpreted as lines in a linear line complex $\mathcal{L}$ form a fundamental line complex $\lsf : \Z^3 \rightarrow \{\mbox{lines in $\R\P^3$}\}$.
\end{definition}

By construction, the two circles of a fundamental circle complex attached to the vertices of any edge of the $\Z^3$ lattice touch in an oriented manner. However, the condition of oriented contact does not entirely characterise fundamental circle complexes. An algebraic characterisation of fundamental circle complexes is contained in the following theorem.

\begin{theorem}
  Fundamental circle complexes are governed by the real symmetric $M$-system (\ref{E12}), that is, $\bar{\mathcal{M}} = \mathcal{M}$ and $\mathcal{M}^T = \mathcal{M}$.
\end{theorem} 

\begin{proof}
In the preceding section, it has been demonstrated that the real symmetric $M$-system gives rise to fundamental circle complexes. Conversely, by definition, fundamental circle complexes are encoded in the real $M$-system subject to $M^{54}=M^{45}$. The latter constraint is now shown to imply that, up to a gauge transformation, $M^{ik} = M^{ki}$ for all $i,k\in\{1,2,3,4,5\}$ and, hence, $\mathcal{M}^T = \mathcal{M}$. Thus, we begin by reformulating the $M$-system as
\bela{E22aa}
 \bear{c}\dis
  \hat{\mathcal{M}}_l = \hat{\mathcal{M}} - \frac{\calV^l{(\calW^l)}^T}{M^{ll}},\quad 
  \calV^l = \left(\bear{c}M^{4l}\\M^{5l}\ear\right),\quad \calW^l = \left(\bear{c}M^{l4}\\M^{l5}\ear\right)\AS\dis
  \calV^m _l= \calV^m - \calV^l\frac{M^{lm}}{M^{ll}},\quad \calW^m_l = \calW^m - \frac{M^{ml}}{M^{ll}}\calW^l\AS\dis
  M^{mn}_l = M^{mn} - \frac{M^{ml}M^{ln}}{M^{ll}},
 \ear
\ela
where all indices are in $\{1,2,3\}$ and $l\not\in \{m,n\}$. The symmetry of $\hat{M}$ implies that
\bela{E22b}
  \calV^l{(\calW^l)}^T = \calW^l{(\calV^l)}^T
\ela
which can only be valid if
\bela{E22c}
  \calV^l = s^l \calW^l
\ela
for some scalar functions $s^l$. The evolution of the above relation gives rise to
\bela{E22d}
  s^l\calW^l - s^m\calW^m\frac{M^{ml}}{M^{mm}} = s^l_m\left(\calW^l - \frac{M^{lm}}{M^{mm}}\calW^m\right)
\ela
and, hence, we arrive at
\bela{E22e}
  s^l_m = s^l,\quad s^l M^{lm} = s^mM^{ml}.
\ela
In particular, we conclude that $s^l = s^l(n_l)$. Since the $M$-system in the form (\ref{E22a}) is readily seen to be invariant under $s^l(\calW^l,M^{lm})\rightarrow (\calW^l,M^{lm})$, we may set $s^l=1$ without loss of generality so that $\calV^l=\calW^l$ and $M^{lm}=M^{ml}$, leading to $\mathcal{M}^T=\mathcal{M}$.
\end{proof}

\subsection{The discrete CKP equation}

Before we embark on establishing a purely Lie geometric characterisation of fundamental circle complexes, we discuss in more detail the algebraic implications of the above theorem. Thus, it is well known (cf.\ \cite{BobenkoSchief15}) that the general $M$-system (\ref{E12}) encapsulates a $\tau$-function which is defined by the compatible system
\bela{E23}
  \tau_i = M^{ii}\tau.
\ela
Iteration of this linear relation leads to
\bela{E24}
 \tau_{ik} = \left|\bear{cc}M^{ii} & M^{ik}\\ M^{ki} & M^{kk}\ear\right|\tau,\quad
 \tau_{ikl} = \left|\bear{ccc} M^{ii} & M^{ik} & M^{il}\\ M^{ki} & M^{kk} & M^{kl}\\ M^{li} & M^{lk} & M^{ll}\ear\right|\tau
\ela
for distinct indices $i,k,l$. In the symmetric case $M^{ik}=M^{ki}$, we may therefore solve for $M^{ii}$ and $M^{ik}$ to obtain
\bela{E25}
  M^{ii} = \frac{\tau_i}{\tau},\quad M^{ik} = \pm\sqrt{\frac{\tau_i\tau_k - \tau\tau_{ik}}{\tau^2}}.
\ela
``Taking the square'' of the remaining relation (\ref{E24})$_3$ then leads to the discrete CKP (dCKP) equation
\bela{E26}
  (\tau\tau_{123} + \tau_1\tau_{23} - \tau_2\tau_{13} - \tau_3\tau_{12})^2 - 4(\tau_{12}\tau_{13} -   \tau_1\tau_{123})(\tau_2\tau_3 - \tau\tau_{23}) = 0
\ela
which, by construction, is (implicitly) symmetric in the indices. In fact, the left-hand side of this equation is nothing but Cayley's $2\times2\times2$ hyperdeterminant \cite{GelfandKapranovZelevinsky94}. The dCKP equation constitutes one of three discrete ``master equations'' in the theory of integrable systems \cite{Schief03b} and has also been obtained by Kashaev \cite{Kashaev96} in the context of star triangle moves in the Ising model. It also appears as a reduction of the hexahedron recurrence proposed by Kenyon and Pemantle \cite{KenyonPemantle13} in connection with cluster algebras and dimer configurations.  Furthermore, the dCKP equation interpreted as a local relation between the principal minors of a symmetric matrix $\calM$ has been derived as a characteristic property by Holtz and Sturmfels \cite{HoltzSturmfels07} in connection with the `principal minor assignment problem'. This remarkable interpretation naturally places this equation on a multi-dimensional lattice and implies its multi-dimensional consistency \cite{TsarevWolf09}. We also observe in passing that, for the algebraic connection between the symmetric $M$-system and the dCKP equation to hold, $M$ does not have to be real and, hence, fundamental line complexes composed of lines of a complex linear line complex, which are encoded in the reduction (\ref{E18}), are governed by the complex dCKP equation. 

If we now introduce quantities $\tau_4,\tau_5$ and $\tau_{45}$ according to
\bela{E27}
  \tau_4 = M^{44}\tau,\quad \tau_5 = M^{55}\tau,\quad \tau_{45} =  \left|\bear{cc}M^{44} & M^{45}\\ M^{45} & M^{55}\ear\right|\tau
\ela
then it may be directly shown that these three quantities constitute another three solutions of the dCKP equation. In fact, this is a consequence of the fact that the general $M$-system may be consistently placed on lattices of arbitrary dimension \cite{BobenkoSchief15} so that one may interpret the indices 4 and 5 on $\tau$ as shifts in two additional lattice directions labelled by $n_4$ and $n_5$. In other words, $\tau_4$ and $\tau_5$ are so-called B\"acklund transforms \cite{BobenkoSuris09,RogersSchief02} of $\tau$ with $\tau_{45}$ being a B\"acklund transform of both $\tau_4$ and $\tau_5$. Thus, we have established the following theorem.

\begin{theorem}
Fundamental circle complexes are parametrised by solutions $\tau$ of the dCKP equation (\ref{E26}) and their first and second generation B\"acklund transforms $\tau_4,\tau_5$ and $\tau_{45}$ respectively defined by (\ref{E27}).
\end{theorem}

\begin{remark}
If we rescale the homogeneous coordinates (\ref{E26}) of the points $\hat{\Vsf}$ in the Lie quadric $Q^3$ then we obtain the explicit parametrisation
\bela{E28}
  \hat{\Vsf} = (\tau,\tau_{45},\tau_4,\tau_5,\pm\sqrt{\tau_4\tau_5 - \tau\tau_{45}})
\ela
of the fundamental circle complexes in terms of the four B\"acklund-related solutions $\tau,\tau_4,\tau_5$ and $\tau_{45}$ of the dCKP equation. However, it is evident that the dCKP equation is not a well-defined evolution equation in the sense that there exist two choices of $\tau_{123}$ for any given $\tau,\tau_1,\tau_2,\tau_3,\tau_{12},\tau_{13},\tau_{23}$. This ambiguity is reflected in the choice of sign in the parametrisation (\ref{E28}) and seen to be key in the geometric definition of fundamental circle complexes. It is important to note that, at the level of the $M$-system, this problem does not occur since the $M$-system constitutes a well-defined system of evolution equations. Accordingly, the $M$-system encapsulates a canonical choice of the root $\tau_{123}$. 
\end{remark}

\subsection{Correlated circles}

As recalled in Section 2, seven lines of an elementary cube of a fundamental line complex uniquely determine the eighth line via the planarity property or, equivalently, the correlation associated with the six lines forming a hexagon. Accordingly, a fundamental line complex is uniquely obtained by prescribing lines on the three coordinate planes $n_i=0$ such that neighbouring lines intersect. It is evident that the same Cauchy problem is valid in the context of  fundamental circle complexes if this map from seven lines to the eighth line acts within the linear line complex $\mathcal{L}$. The latter turns out to be the case and, in geometric terms, this means that if we begin with a ``hexagon'' of oriented circles $(\csf_1,\csf_{12},\csf_{2},\csf_{23},\csf_3,\csf_{13})$ which touch each other cyclically (cf.\  Figure \ref{pairs_of_circles}) then there exists a canonical correlation between the two pairs of oriented circles $\csf$ and $\csf_{123}$ touching the triples of circles $\csf_1,\csf_2,\csf_3$ and $\csf_{12},\csf_{23},\csf_{13}$ respectively.
\begin{figure}
  \centerline{\includegraphics[scale=0.25]{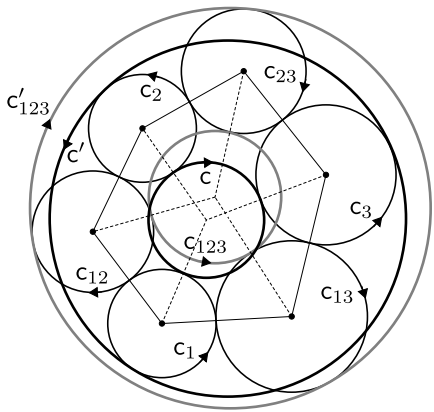}}
  \caption{Apollonius problems associated with a hexagon of touching circles. The pairs of circles $(\csf,\csf')$ and $(\csf_{123},\csf_{123}')$ are in oriented contact with the triple of circles $\csf_1,\csf_2,\csf_3$ and $\csf_{12},\csf_{23},\csf_{13}$ respectively.}
\label{pairs_of_circles}
\end{figure}
Here, we make the assumption that the position of the two triples of circles is such that the two corresponding Apollonius problems are solvable. It is recalled that the classical Apollonius problem  is concerned with the determination of all circles which touch any given triple of distinct circles \cite{Coolidge16}. It turns out that the solvability of one Apollonius problem implies the solvability of the other. Hence, if the Cauchy data on the coordinate planes are chosen in such a manner that, wherever applicable, the Apollonius problems are solvable then the solvability of the Apollonius problem propagates along the lattice. The above assertions are a consequence of the following key theorem.

\begin{theorem}
The correlation associated with a hexagon of six intersecting lines which belong to a linear line complex $\mathcal{L}$ acts within $\mathcal{L}$. In particular, seven oriented circles of an elementary cube of a circle complex which touch each other ``along edges'' are mapped to a unique eighth circle which has oriented contact with the associated triple of circles along the remaining three edges.
\end{theorem}

\begin{proof}
The general linear line complex $\mathcal{L}$ is defined by \cite{SempleKneebone52}
\bela{E29}
  \sum_{i\neq k}\lambda^{ik}\gamma^{ik} = 0,
\ela
where $\Lambdasf = (\lambda^{ik})_{i,k}$ is a real anti-symmetric $4\times4$ matrix. In terms of the parametrisation (\ref{E6}) of the Pl\"ucker coordinates $\gamma^{ik}$, this condition may be formulated as \cite{SempleKneebone52}
\bela{E30}
  \asf^T\Lambdasf\bsf = 0.
\ela
Hence, the lines of $\mathcal{L}$ are null lines with respect to the skew-symmetric bilinear form defined by $\Lambdasf$, that is, $\lsf\in\mathcal{L}$ if and only if two distinct points (and therefore any two points) $\asf$ and $\bsf$ on $\lsf$ are orthogonal with respect to $\Lambdasf$. 

We now consider a spatial hexagon $[\xsf^1,\xsf^2,\xsf^3,\xsf^4,\xsf^5,\xsf^6]$ (cf.\ Figure \ref{hexagon}) in general position with associated lines (edges) $\lsf^{i i+1}$ as indicated in Figure \ref{correlation}. On the one hand, the existence of a corresponding correlation $\kappa$ encoded in a symmetric matrix $\Bsf$ implies that
\bela{E31}
 {\xsf^{i+2}}^T\Bsf\xsf^i = 0,\quad {\xsf^{i+3}}^T\Bsf\xsf^i = 0,\quad{\xsf^{i+4}}^T\Bsf\xsf^i = 0. 
\ela
Furthermore, if we assume that the lines $\lsf^{ii+1}$ belong to a linear line complex $\mathcal{L}$ then
\bela{E32}
 {\xsf^{i+2}}^T\Lambdasf\xsf^{i+3} = 0,\quad {\xsf^{i+3}}^T\Lambdasf\xsf^{i+3} = 0,\quad{\xsf^{i+4}}^T\Lambdasf\xsf^{i+3} = 0. 
\ela
Hence, we conclude that
\bela{E33}
  \Lambdasf\xsf^{i+3} \sim \Bsf \xsf^i
\ela
which, in turn, leads to
\bela{E34}
  \xsf^i \sim (\Lambdasf^{-1}\Bsf)^2\xsf^i
\ela
provided that the skew-symmetric bilinear form defined by $\Lambdasf$ is non-degenerate. Since the identity (\ref{E34}) holds for all $\xsf^i$, we have established that
\bela{E35}
  (\Lambdasf^{-1}\Bsf)^2 \sim\one\quad\Rightarrow\quad \Bsf\Lambdasf^{-1}\Bsf \sim \Lambdasf.
\ela

For any line $\lsf\in\mathcal{L}$ represented by two distinct points $\ysf^1$ and $\ysf^2$, we may define
\bela{E36}
  \ysf^{\mu\prime} = \Lambdasf^{-1}\Bsf\ysf^\mu 
\ela
and readily see that
\bela{E37}
  {\ysf^{\mu\prime}}^T\Bsf\ysf^\nu = - {\ysf^\mu}^T\Bsf\Lambdasf^{-1}\Bsf\ysf^\nu \sim {\ysf^{\mu}}^T\Lambdasf\ysf^\nu = 0.
\ela
Hence, the line $\lsf'$ passing through the points $\ysf^{1\prime}$ and $\ysf^{2\prime}$ is the image of the line $\lsf$ under the correlation $\kappa$. Furthermore,
\bela{E38}
  {\ysf^{\mu\prime}}^T\Lambdasf\ysf^{\nu\prime} = -{\ysf^\mu}^T\Bsf\Lambdasf^{-1}\Bsf\ysf^\nu \sim{\ysf^\mu}^T\Lambdasf\ysf^\nu = 0
\ela
so that the line $\lsf'$ is an element of the linear line complex $\mathcal{L}$ as asserted. In particular, a line $\lsf\in\mathcal{L}$ which intersects the lines $\lsf^{12},\lsf^{34}$ and $\lsf^{56}$ is mapped to a line $\lsf'\in\mathcal{L}$ which intersects the lines $\lsf^{23},\lsf^{45}$ and $\lsf^{61}$. This proves the second part of the theorem.
\end{proof}

\begin{remark}
The proof of the above theorem implies that the projective transformation $\xsf \mapsto \Lambdasf^{-1}\Bsf\xsf$ maps opposite vertices $\xsf^i$ and $\xsf^{i+3}$ to each other. In fact, in this sense, a hexagon may be mapped to itself if and only if the edges of the hexagon belong to a non-degenerate linear line complex. This is seen as follows.
We consider a spatial hexagon $[\xsf^1,\xsf^2,\xsf^3,\xsf^4,\xsf^5,\xsf^6]$ in general position and its associated correlation represented by a symmetric matrix $\Bsf$. Furthermore, we assume that there exists a projective transformation $\xsf\mapsto\Csf\xsf$ which interchanges opposite vertices of the hexagon, that is,
\bela{E38a}
 \xsf^{i+3} \sim \Csf\xsf^i,
\ela
where, as usual, indices are taken modulo 6. If we define a matrix $\Lambdasf$ by
\bela{E38b}
  \Lambdasf = \Bsf\Csf^{-1}
\ela
then
\bela{E38c}
  {\xsf^i}^T\Lambdasf\xsf^i = {\xsf^i}^T\Bsf\Csf^{-1}\xsf^i\sim {\xsf^i}^T\Bsf\xsf^{i+3} = 0.
\ela
In an analogous manner, we deduce that
\bela{E38d}
     {\xsf^i}^T\Lambdasf\xsf^{i+1} =  {\xsf^{i+1}}^T\Lambdasf\xsf^i = 0.
\ela
Accordingly. if $\Lambdasf$ is skew-symmetric then the (extended) edges $[\xsf^i,\xsf^{i+1}]$ of the hexagon belong to the linear line complex defined by $\Lambdasf$. On the other hand, if $\Lambdasf^s$ denotes the symmetric part of $\Lambdasf$ then we conclude from (\ref{E38c}) and (\ref{E38d}) that
\bela{E38e}
  {\xsf^i}^T\Lambdasf^s\xsf^{i} = {\xsf^i}^T\Lambdasf^s\xsf^{i+1} = 0.
\ela
However, this implies that $\Lambdasf^s=0$ as may be shown by choosing, without loss of generality, $\xsf^1,\ldots,\xsf^4$ as the standard orthonormal basis of $\R^4$ and evaluating the conditions (\ref{E38e}).
\end{remark}

The above theorem gives rise to the following natural definition and characterisation of fundamental circle complexes.

\begin{definition}\label{correlation_def}
Two ``opposite'' circles of an elementary cube of eight oriented circles which touch along edges are {\em correlated} if the two circles are images of each other under the correlation associated with the hexagon of the remaining six circles.
\end{definition}

\begin{corollary}
A circle complex of oriented circles which touch along edges is fundamental if and only if opposite circles of any elementary cube of circles are correlated.
\end{corollary}

\subsection{The Lie geometry of linear line complexes}

For a completely geometric characterisation of fundamental circle complexes, it is now necessary to determine which of the two oriented circles $\csf_{123}$ which touch the triple of circles $\csf_{12},\csf_{23},\csf_{13}$ belonging to a hexagon of oriented circles $[\csf_1,\csf_{12},\csf_{2},\csf_{23},\csf_3,\csf_{13}]$ is correlated to any of the two oriented circles $\csf$ which touch the other triple of circles $\csf_1,\csf_2,\csf_3$. To this end, we first provide a Lie geometric characterisation of a line in $\R\P^3$ which is not in a given linear line complex
\bela{E39}
  \mathcal{L} = \{\lsf\subset\C\P^3 : \ysf^T\Lambdasf\xsf = 0\quad \forall [\xsf],[\ysf]\in\lsf\}.
\ela
Any point in $\R\P^3$ with homogeneous coordinates $\asf$ gives rise to a unique plane in $\R\P^3$ which is spanned by the one-parameter family of lines $\mathcal{L}^\asf\subset \mathcal{L}$ passing through $\asf$. This plane is represented by the orthogonal complement
\bela{E40}
  \asf^\perp = \{\xsf : \xsf^T\Lambdasf\asf = 0\}.
\ela
Any two lines $\lsf^\asf,\lsf^\asf_\ast\in\mathcal{L}^\asf$ correspond to two points in the Lie quadric $Q^3\subset Q^4$ which, in terms of homogeneous coordinates, have vanishing inner product, that is, 
\bela{E41}
  \langle\Vsf^\asf,\Vsf^\asf_\ast\rangle = 0
\ela
so that $\mathcal{L}^\asf$ represents a null line in $Q^3$ or, equivalently, a pencil $\mathcal{C}^\asf$ of touching oriented circles  in the plane. We will refer to the latter as a {\em contact element} in the plane. Accordingly, a line in $\R\P^3$ regarded as a set of points gives rise to a one-parameter family of contact elements in the plane. In the classical literature, the latter is termed a {\em turbine} (`Turbine') \cite{Strubecker53}. In particular, the turbine $\mathcal{T}^{\lsf}$ associated with a line $\lsf\in\mathcal{L}$ consists of all contact elements which contain the corresponding circle $\csf$, that is, it captures all circles which touch $\csf$. As indicated above, we now examine the geometry of the contact elements associated with lines which are not contained in the linear line complex $\mathcal{L}$. 

We consider two points $\asf$ and $\bsf$ on a given line $\lsf\not\in\mathcal{L}$ which we denote by
\bela{E42}
  \lsf = [\asf,\bsf]\not\in\mathcal{L}.
\ela
The line of intersection $\lsf'$ of the two planes represented by $\asf^\perp$ and $\bsf^\perp$ is independent of $\asf$ and $\bsf$ and is given by the orthogonal complement
\bela{E43}
  \lsf' = \lsf^\perp = \{[\ysf] : \ysf^T\Lambdasf\xsf=0\quad\forall [\xsf]\in\lsf\}
\ela
of $\lsf$. By definition of the orthogonal complement, the line connecting any two points on $\lsf$ and $\lsf'$ belongs to the linear line complex $\mathcal{L}$. Hence, if we choose two points $\asf'$ and $\bsf'$ on $\lsf'$ then the points $\asf,\asf',\bsf,\bsf'$ are the vertices of a quadrilateral, the (extended) edges of which belong to $\mathcal{L}$ (cf.\ Figure \ref{complement}).
\begin{figure}
\centerline{\includegraphics[scale=0.5]{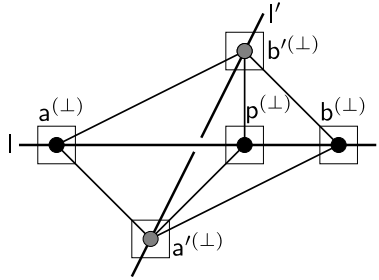}}
\caption{A line $\lsf$, its orthogonal complement $\lsf'=\lsf^\perp$ and the connecting lines of a linear line complex $\mathcal{L}$.}
\label{complement}
\end{figure}
In Lie geometric terms, this corresponds to four cyclically touching circles $\csf^{[\asf,\asf']}$,  $\csf^{[\asf',\bsf]}$,  $\csf^{[\bsf,\bsf']}$ and $\csf^{[\bsf',\asf]}$ as indicated in Figure \ref{turbine}. It is evident that the pairs of circles $\csf^{[\asf,\asf']},\,\csf^{[\asf',\bsf]}$ and 
$\csf^{[\bsf,\bsf']},\,\csf^{[\bsf',\asf]}$ generate the contact elements $\mathcal{C}^{\asf'}$ and $\mathcal{C}^{\bsf'}$ respectively. In general, the turbine $\mathcal{T}^{\lsf'}$ associated with the line $\lsf'$ constitutes the collection of contact elements which are generated by pairs of touching circles contained in the contact elements $\mathcal{C}^{\asf}$ and $\mathcal{C}^{\bsf}$ respectively. If, in the sense of M\"obius geometry, we regard points in the plane as Lie circles of ``zero radius'' defined by the intersection of the Lie quadric with a hyperplane of signature $(3,1)$ then the circles of zero radius contained in the turbine $\mathcal{T}^{\lsf'}$ form a unique circle as indicated in Figure~\ref{turbine}. 
\begin{figure}
\centerline{\includegraphics[scale=0.5]{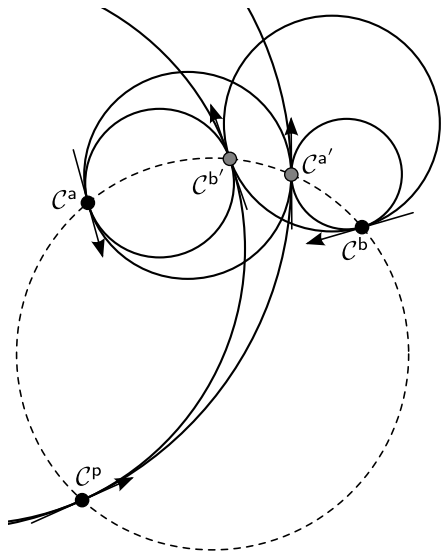}}
\caption{The Lie geometric avatar of Figure \ref{complement}: the ``location'' of the contact elements of the turbine $\mathcal{T}^{\lsf}$ and its ``orthogonal complement'' $\mathcal{T}^{\lsf'}$ with $\lsf=[\asf,\bsf]$ and $\lsf^\perp=\lsf'=[\asf',\bsf']$. }
\label{turbine}
\end{figure}
This is due to the fact that the points of contact of four cyclically touching circles are concyclic. Accordingly, we can think of the contact elements of $\mathcal{T}^{\lsf'}$ as being ``located'' on a circle. It is important to note that this circle is not a circle in the sense of Lie geometry but merely serves as an aid to visualise a turbine.

The turbine $\mathcal{T}^{\lsf}$ is now constructed by choosing an arbitrary point $\psf$ on the line $\lsf$ and noting that the lines connecting $\psf$ and $\lsf'$ correspond to the touching circles which form the contact element $\mathcal{C}^{\psf}$. In particular, the lines $[\psf,\asf']$ and $[\psf,\bsf']$ are associated with circles $\csf^{[\psf,\asf']}\in \mathcal{C}^{\asf'}$ and $\csf^{[\psf,\bsf']}\in \mathcal{C}^{\bsf'}$ which generate $\mathcal{C}^{\psf}$ as depicted in Figure \ref{turbine}. Thus, the turbine $\mathcal{T}^{\lsf}$ constitutes the collection of contact elements which are generated by pairs of touching circles contained in the contact elements $\mathcal{C}^{\asf'}$ and $\mathcal{C}^{\bsf'}$ respectively. It is evident that $\mathcal{T}^{\lsf}$ is independent of the choice of $\asf'$ and $\bsf'$ so that we may refer to $\mathcal{T}^{\lsf}$ as being generated by the contact elements $\mathcal{C}^\asf$ and $\mathcal{C}^\bsf$. Once again, in the above-mentioned sense, the contact elements of $\mathcal{T}^{\lsf}$ are located  on a unique circle which coincides with the circle attached to the turbine $\mathcal{T}^{\lsf'}$. We are now in a position to give a M\"obius geometric construction of correlated circles which is shown to be meaningful in a Lie geometric sense by interpreting circles which are not Lie circles as turbines.

\subsection{A Lie geometric characterisation of correlated circles}

We begin by reformulating the construction of an elementary cube of a fundamental line complex in terms of the concurrency property (cf.\ Definition \ref{line_complex}). To this end, we denote by 
\bela{E44}
  \lsf^{l,m} = [\psf^l,\psf^l_m],\quad l\neq m
\ela
the ``diagonal'' passing through the points $\psf^l$ and $\psf^l_m$ as depicted in Figure \ref{diagonals}.
\begin{figure}
\centerline{\includegraphics[scale=0.5]{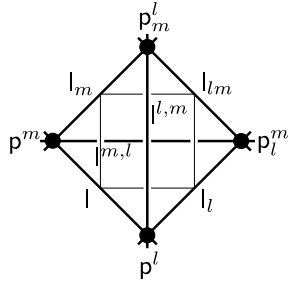}}
\caption{The relationship between the lines $\lsf$, the points of intersection $\psf^l$ and the diagonals $\lsf^{m,p}$.}
\label{diagonals}
\end{figure}
The four diagonals of the same ``type'' $\lsf^{l,m},\lsf^{l,m}_n,\lsf^{n,m},\lsf^{n,m}_l$ for distinct $l,m,n$ meet in a point which is labelled by $\psf^{(m)}$ (cf.\ Figure \ref{fundamental}).

\begin{theorem}
Let $\lsf,\lsf_1,\lsf_2,\lsf_3,\lsf_{12},\lsf_{23},\lsf_{13}$ be lines in $\R\P^3$ which are combinatorially attached to seven vertices of an elementary cube and intersect along edges in the points $\psf^l$ and $\psf^l_m$. The points $\psf^{(m)}$ are defined as the points of intersection of the pairs of diagonals $\lsf^{l,m},\lsf^{n,m}$ for distinct $l,m,n$. The pairs of lines $\lsf^{l,m}_n=[\psf^l_n,\psf^{(m)}]$ and $\lsf^{l,n}_m=[\psf^l_m,\psf^{(n)}]$ are coplanar and define the points of intersection $\psf^l_{mn}$. Then, the points $\psf^1_{23},\psf^2_{13}$ and $\psf^3_{12}$ are collinear and define an eighth line $\lsf_{123}$ so that the elementary cube of eight lines is fundamental. Moreover, if the given seven lines are in a linear line complex $\mathcal{L}$ then so is~$\lsf_{123}$.
\end{theorem}

On use of the following dictionary, the above theorem may now be directly translated into the language of Lie geometry.
\begin{center}
    \begin{tabular}{ | p{5cm} | p{5cm} |}
    \hline
    linear line complex geometry & Lie circle geometry\\ \hline\hline
    line $\lsf$ & circle $\csf$\\ \hline
    point of intersection $\psf^l$ & contact element $\mathcal{C}^l$\\ \hline
    diagonal  $\lsf^{l,m}$ & turbine $\mathcal{T}^{l,m}$\\ \hline
    point of concurrency $\psf^{(l)}$ & common contact element $\mathcal{C}^{(l)}$\\ 
   \hline
    \end{tabular}
\end{center}
\noindent
It is noted that the turbine $\mathcal{T}^{l,m}$ is generated by the contact elements $\mathcal{C}^l$ and~$\mathcal{C}^l_m$.

\begin{corollary}\label{corollary}
Let $\csf,\csf_1,\csf_2,\csf_3,\csf_{12},\csf_{23},\csf_{13}$ be oriented circles in $\R^2$ which are combinatorially attached to seven vertices of an elementary cube and touch along edges, thereby defining the contact elements $\mathcal{C}^l$ and $\mathcal{C}^l_m$. The turbines $\mathcal{T}^{l,m}$ and $\mathcal{T}^{n,m}$ for distinct $l,m,n$ share a contact element denoted by $\mathcal{C}^{(m)}$. The turbines $\mathcal{T}^{l,m}_n$ and $\mathcal{T}^{l,n}_m$ generated by the pairs of contact elements $\mathcal{C}^l_n,\mathcal{C}^{(m)}$ and $\mathcal{C}^l_m,\mathcal{C}^{(n)}$ respectively have a contact element $\mathcal{C}^l_{mn}$ in common. Then, there exist a unique circle $\csf_{123}$ which is contained in the contact elements  $\mathcal{C}^1_{23},\mathcal{C}^2_{13}$ and $\mathcal{C}^3_{12}$. By construction, the circles $\csf_{123}$ and $\csf$ are correlated and the elementary cube of eight circles is fundamental.
\end{corollary}

In view of the following theorem, we refer to four oriented circles $\csf^a,\csf^b,\csf^c,\csf^d$ which touch each other cyclically and therefore generate contact elements\linebreak $\mathcal{C}^{ab},\mathcal{C}^{bc}, \mathcal{C}^{cd},\mathcal{C}^{da}$ as a `quadrilateral of circles' $[\csf^a,\csf^b,\csf^c,\csf^d]$. The two turbines $\mathcal{T}^{ab,cd}$ and $\mathcal{T}^{bc,da}$ generated by the pairs of contact elements $\mathcal{C}^{ab},\mathcal{C}^{cd}$ and $\mathcal{C}^{bc},\mathcal{C}^{da}$ respectively constitute the ``diagonals'' of the quadrilateral. Moreover, for any pair of contact elements $\mathcal{C}\in\mathcal{T}^{ab,cd}$ and $\mathcal{C}'\in\mathcal{T}^{bc,da}$, there exists a unique circle contained in both contact elements $\mathcal{C}$ and $\mathcal{C}'$. In the above-mentioned M\"obius geometric interpretation, the two turbines $\mathcal{T}^{ab,cd}$ and $\mathcal{T}^{bc,da}$ are represented by the same circle $\csf^{abcd}$ which passes through the ``points of contact'' $p^{ab},p^{bc},p^{cd},p^{da}$ of the four circles.

\begin{theorem}\label{tim}
Given a hexagon $[\csf_1,\csf_{12},\csf_{2},\csf_{23},\csf_3,\csf_{13}]$ of cyclically touching oriented circles and one of the two circles $\csf$ which touch the triple of circles $\csf_1,\csf_2,\csf_3$, we denote by $\csf^{lm},\,l\neq m$ the circle passing through the points of contact $p^l,p^m,p^l_m,p^m_l$ (in the M\"obius geometric sense) of the quadrilateral of circles $[\csf,\csf_l,\csf_{lm},\csf_{m}]$  (cf.\ Figure~\ref{eight_circle}). The circles $\csf^{lm}$ and $\csf^{nm}$ intersect in the additional point $p^{(m)}$ for distinct $l,m,n$. There exist circles $\csf^{lm}_n$ which pass through the two points of contact $p^l_n,p^m_n$ of the circles $\csf_n,\csf_{ln},\csf_{mn}$ and the points $p^{(l)},p^{(m)}$. The circles $\csf^{lm}_n,\csf^{ln}_m$ and $\csf_{mn}$ meet in a point $p^l_{mn}$. Then, the (appropriately oriented) circle $\csf_{123}$ passing through the points $p^1_{23},p^2_{13}$ and $p^3_{12}$ touches the triple of circles $\csf_{12},\csf_{23},\csf_{13}$, is well-defined in a Lie geometric sense and is correlated to the circle $\csf$ in the sense of Definition \ref{correlation_def}.
\end{theorem}

\begin{figure}
  \centerline{\includegraphics[scale=0.35]{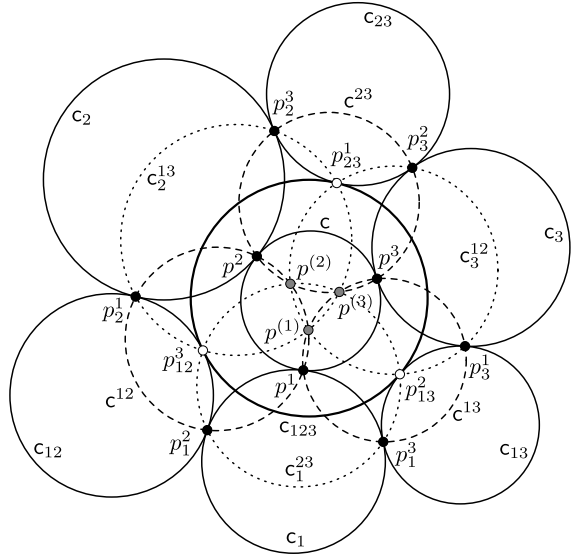}}
  \caption{The M\"obius geometric construction of pairs of correlated circles $\csf,\csf_{123}$.}
\label{eight_circle}
\end{figure}

\begin{proof}
In the context of M\"obius geometry, this theorem may be proven by iteratively applying (degenerations of) Miquel's theorem \cite{Pedoe88}. Here, we present a proof which demonstrates that the circle $\csf_{123}$ not only exists but also admits a Lie geometric interpretation. Thus, the turbines $\mathcal{T}^{l,m}$ related to the quadrilaterals $[\csf,\csf_l,\csf_{lm},\csf_{m}]$ are represented by the (dashed) circles $\csf^{lm}$ passing through the points of contact $p^l,p^m,p^l_m,p^m_l$ as indicated in Figure \ref{eight_circle}. The turbines $\mathcal{T}^{l,m}$ and $\mathcal{T}^{n,m}$ share a contact element $\mathcal{C}^{(m)}$ which is ``located'' at the point of intersection $p^{(m)} \subset \csf^{lm}\cap\csf^{nm}$. Now, Corollary \ref{corollary} implies that the turbines $\mathcal{T}^{l,m}_n$ and $\mathcal{T}^{m,l}_n$ are the diagonals of a quadrilateral of circles with ``vertices'' $\mathcal{C}^l_n,\mathcal{C}^{(m)}$ and $\mathcal{C}^m_n,\mathcal{C}^{(l)}$. The (dotted) circle $\csf^{lm}_n$ representing these two turbines passes through the points of contact $p^l_n,p^m_n$ and $p^{(l)},p^{(m)}$. The turbines $\mathcal{T}^{l,m}_n$ and $\mathcal{T}^{l,n}_m$ share a contact element $\mathcal{C}^l_{mn}$ which contains the circle $\csf_{mn}$. Hence, the circles $\csf^{lm}_n,\csf^{ln}_m$ and $\csf_{mn}$ meet at a point $p^l_{mn}$. Since there exists a unique circle $\csf_{123}$ contained in the contact elements $\mathcal{C}^1_{23},\mathcal{C}^2_{13}$ and $\mathcal{C}^3_{12}$ which therefore touches the triple of circles $\csf_{12},\csf_{23},\csf_{13}$, this (bold) circle passes through the points $p^1_{23},p^2_{13}$ and $p^3_{12}$ in the M\"obius geometric sense.
\end{proof} 

\section{M\"obius geometry}

We now embark on a discussion of the existence of fundamental complexes in subgeometries of Lie circle geometry. Here, we consider the intersection of the Lie quadric $Q^3$ with a hyperplane of signature $(3,1)$ leading to a two-dimensional quadric $Q^2$. The points in this ``M\"obius quadric'' may be interpreted as circles of ``zero radius'', that is, points in the plane $\R^2$ \cite{Blaschke29}. (Oriented) circles in the plane are represented by points in the complement $Q^3\backslash Q^2$. A point $\psf$ lies on a circle $\csf$ if and only if the inner product of their homogeneous coordinates $\Vsf_\bullet$ and $\Vsf_\ocircle$ respectively vanishes, that is,
\bela{E46}
  \langle\Vsf_\bullet,\Vsf_\ocircle\rangle = 0.
\ela
It is therefore natural to examine fundamental circle complexes which are such that every second Lie circle constitutes a point in the sense of M\"obius geometry. It turns out that the orientation of the other half of the Lie circles is insignificant and may be introduced in a consistent manner for any given fundamental complex of points and non-oriented circles. Hence, the fundamental point-circle complexes defined below may be interpreted as particular fundamental circle complexes and therefore correspond to a subclass of solutions of the symmetric $M$-system (\ref{E12}).  

\begin{definition}
A configuration of points and circles combinatorially attached to the vertices of the even and odd sublattices of a $\Z^3$ lattice respectively is termed a {\em fundamental point-circle complex} if the points and circles are incident along the edges of the $\Z^3$ lattice.
\end{definition}

\subsection{Geometric construction of fundamental point-circle complexes}

In order to determine a canonical Cauchy problem for fundamental point-circle complexes, we represent a fundamental point-circle complex $(\psf,\csf)$ by two maps
\bela{E47}
 \bear{rl}
  \psf : \De_3 & \rightarrow \{\mbox{points in $\R^2$}\}\as
  \csf : \Do_3 & \rightarrow \{\mbox{circles in $\R^2$}\},
 \ear
\ela
where
\bela{E48}
 \bear{l}
  \De_3 = \{(n_1,n_2,n_3)\in\Z^3 : n_1 + n_2 + n_3\mbox{ even}\}\as
  \Do_3 = \{(n_1,n_2,n_3)\in\Z^3 : n_1 + n_2 + n_3\mbox{ odd}\}
 \ear
\ela
denote the even and odd sublattices of $\Z^3$ respectively. It is immediately verified that the prescription of points and circles on the coordinate planes $n_i=0$ as Cauchy data is inadmissible since there exist two types of elementary cubes of a fundamental point-circle complex as depicted in Figure \ref{cubes}.
\begin{figure}
\centerline{\includegraphics[scale=0.5]{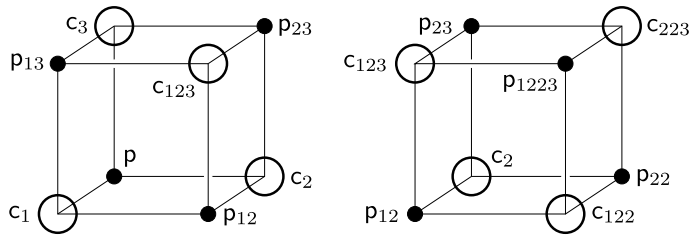}}
\caption{The two types of elementary cubes of a fundamental point-circle complex.}
\label{cubes}
\end{figure}
Indeed, given four points and three circles such as $\psf,\psf_{12},\psf_{23},\psf_{13}$ and $\csf_1,\csf_2,\csf_3$ which respect incidence, the fourth circle $\csf_{123}$ is uniquely determined by the points $\psf_{12},\psf_{23},\psf_{13}$. However, in general, three points and four circles do not define the remaining fourth point, that is, for instance, $\csf_{123}$ and the additional Cauchy data $\csf_{223},\csf_{122}$ are, in general, not concurrent and, hence, the point $\psf_{1223}$ does not exist. 

We proceed by decomposing the lattice $\Z^3 = \De_3\cup\Do_3$ into two-dimensional slices of $A$ type, that is,
\bela{E49}
 \Z^3 = \bigcup_{k\in\Z} A^k_2,\quad A^k_2 =  \{(n_1,n_2,n_3)\in\Z^3 : n_1 + n_2 + n_3 = k\}
\ela
so that (the image of) each $A_2$ slice ``contains'' either points or circles, depending on whether $k$ is even or odd respectively (cf.\ Figure \ref{slices}). 
\begin{figure}
\centerline{\includegraphics[scale=0.5]{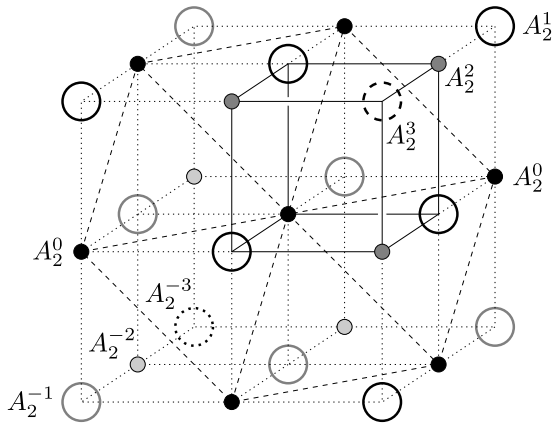}}
\caption{Points and circles associated with the slices $A^k_2$.}
\label{slices}
\end{figure}
Any fundamental point-circle complex $(\psf,\csf)$ may be reconstructed from the points $\psf(A^0_2)$. If we regard $A^0_2$ as a triangular lattice then there exist two types of triangles $\nabla$ and $\Delta$, the vertices of which are mapped to, for instance, $\psf(\nabla) = (\psf,\psf_{\ob3},\psf_{\tb3})$ and $\psf(\Delta) = (\psf,\psf_{1\hb},\psf_{2\hb})$ respectively. The circles passing through these triples of points are given by $\csf(\nabla)=\csf_3$ and $\csf(\Delta) = \csf_{\hb}$. Hence, the circles $\csf(A^1_2)$ and $\csf(A^{-1}_2)$ have been retrieved. The points $\psf(A^2_2)$ and $\psf(A^{-2}_2)$ are then the points of concurrency of appropriate triples of circles in 
 $\csf(A^1_2)$ and $\csf(A^{-1}_2)$ respectively. Specifically, for instance, $\psf_{23} = \csf_2\cap \csf_3\cap\csf_{\ob23}$ and $\psf_{\tb\hb} = \csf_{\tb}\cap\csf_{\hb}\cap_{1\tb\hb}$. However, it is evident that their existence imposes constraints on the points $\psf(A^0_2)$. This process may be continued to reconstruct the circles $\csf(A^3_2),\csf(A^{-3}_2)$ and points $\psf(A^4_2),\psf(A^{-4}_2)$ and, in fact, all remaining points and circles. It turns out that, remarkably, the existence of the points $\psf(A^4_2),\psf(A^{-4}_2)$ does not impose any further constraints on the points $\psf(A^0_2)$. This may be exploited to formulate a canonical Cauchy problem for fundamental point-circle complexes.

\subsubsection{The evolution of constrained two-dimensional Cauchy data}

\begin{theorem}\label{theorem_cauchy}
A fundamental point-circle complex $(\psf,\csf)$ is uniquely determined by constrained Cauchy data $\psf(A^{-2}_2)$, $\csf(A^{-1}_2)$, $\psf(A^0_2)$, $\csf(A^1_2)$ and $\psf(A^2_2)$ which respect incidence (cf.\ Figure \ref{cauchy_circle}). Hence, if points $\psf(A^0_2)$ are chosen in such a manner that the points $\psf(A^{-2}_2)$  and $\psf(A^2_2$) exist then an associated fundamental point-circle complex exists and is unique.
\end{theorem}
\begin{figure}
\centerline{\includegraphics[scale=0.5]{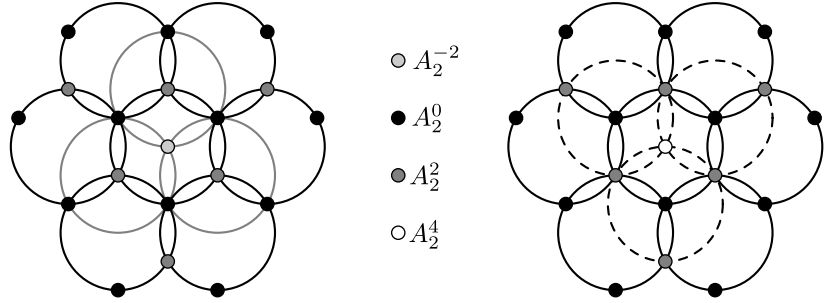}}
\caption{Constrained Cauchy data (left) for a fundamental point-circle complex and their propagation (right) along the $\Z^3$ lattice.}
\label{cauchy_circle}
\end{figure}

\begin{proof}
Given constrained Cauchy data $\csf(A^{-1}_2)$, $\psf(A^0_2)$, $\csf(A^1_2)$ and $\psf(A^2_2)$, the circles $\csf(A^3_2)$ (dashed in Figure \ref{cauchy_circle}) are uniquely determined. The key ``12-circle theorem'' given below then guarantees that the points  $\psf(A^4_2)$ exist if and only if the points $\psf(A^{-2}_2)$ exist. Thus, if the latter are included in the Cauchy data then the points $\psf(A^4_2)$ exist. This process may now be iterated in both forward and backward directions to cover the entire $\Z^3$ lattice.
\end{proof}

The above proof relies on the observation that given the six black ``outer'' circles in Figure \ref{cauchy_circle} then the three dashed circles are concurrent if and only if the three grey circles are concurrent. It is noted that this assertion is true regardless of whether the ``central'' black circle exists or not. This is stated in the following theorem.

\begin{theorem}\label{rhombic_theorem}
Given six cyclically intersecting black circles as in Figure \ref{twelvecircles}, another six grey and dashed circles passing through black and grey points of intersection respectively may be introduced in the combinatorial manner of Figure~\ref{cauchy_circle}. Then, the three grey circles are concurrent if and only if the three dashed circles are concurrent.
\end{theorem}
\begin{figure}
\centerline{\includegraphics[scale=0.4]{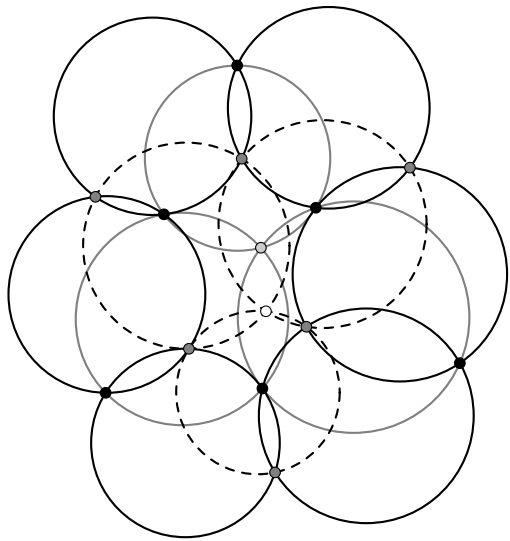}}
\caption{A rhombic dodecahedral point-circle configuration.}
\label{twelvecircles}
\end{figure}

\begin{proof}
The 14 points and 12 circles of the above configuration have the combinatorics of the 14 vertices and 12 quadrilateral faces of a rhombic dodecahedron (cf.\ Figure \ref{dodecahedron}). 
\begin{figure}
\centerline{\includegraphics[scale=0.5]{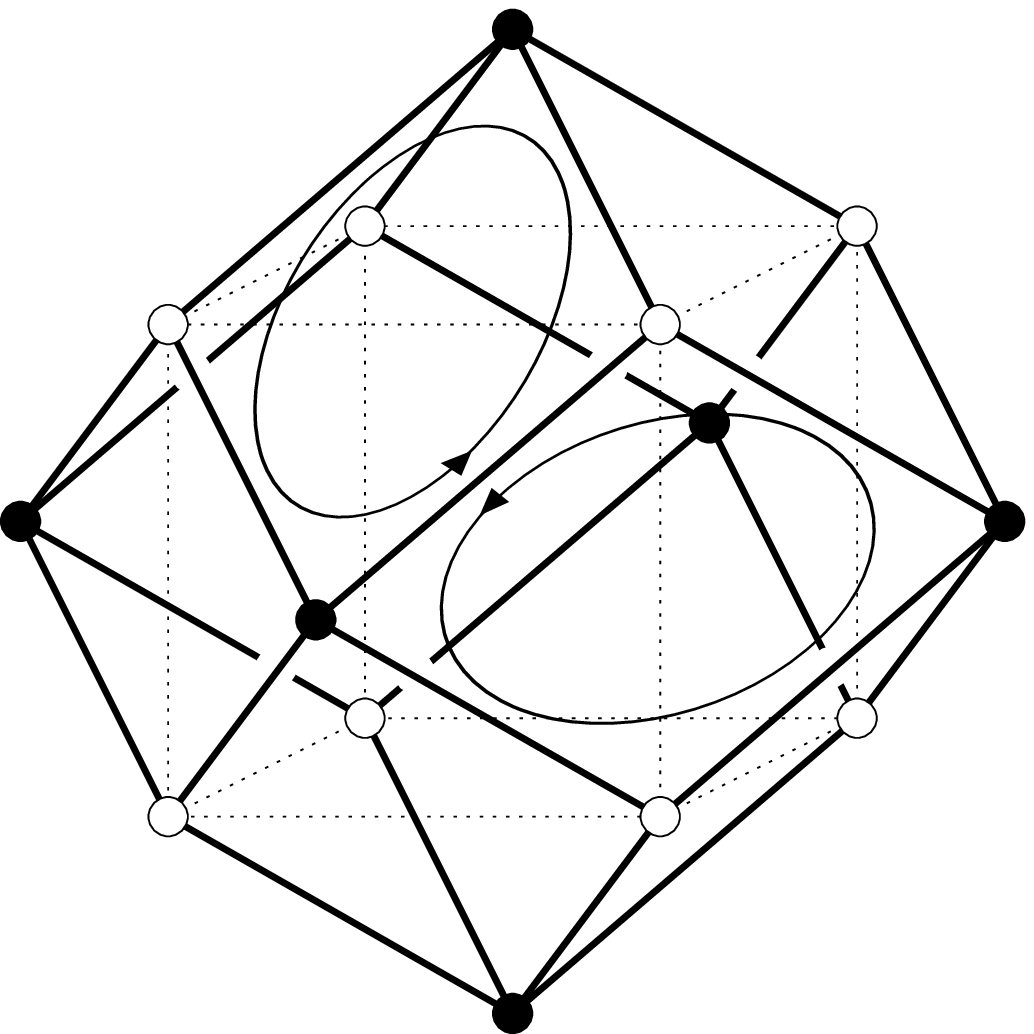}\qquad\qquad\qquad\includegraphics[scale=0.5]{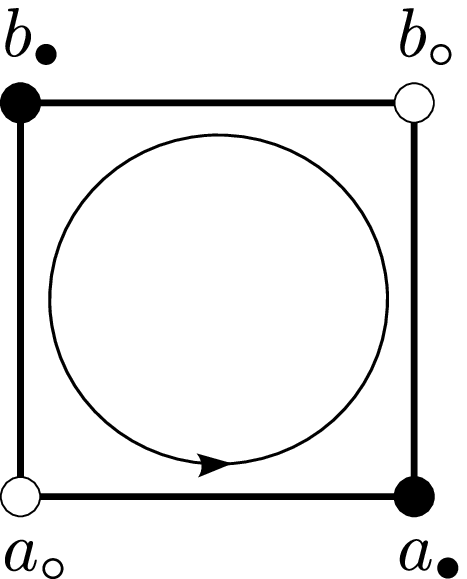}}
\caption{The oriented faces of a rhombic dodecahedron.}
\label{dodecahedron}
\end{figure}
For brevity, we refer to a face as being ``circular'' if the corresponding four points of the configuration are concyclic. Hence, the above theorem may be reformulated in the following manner: if 11 faces of a rhombic dodecahedral configuration of 14 points are circular then so is the remaining face. Since a rhombic dodecahedron is oriented and bipartite, we can associate a cross-ratio with the faces of the configuration. Thus, if we colour the vertices of an oriented face black and white as in Figure \ref{dodecahedron} then the cross-ratio of the face is defined as the cross-ratio
\bela{E50}
  \mathsf{cr} = \frac{(a_\circ - a_\bullet)(b_\circ - b_\bullet)}{(a_\bullet - b_\circ)(b_\bullet - a_\circ)}
\ela
of the cycle $a_\circ\rightarrow a_\bullet\rightarrow b_\circ\rightarrow b_\bullet\rightarrow a_\circ$ of points regarded as complex numbers. It is noted that $\mathsf{cr}$ is well-defined since the cycle which starts at $b_\circ$ leads to the same cross-ratio. It is now evident that the orientation of the faces of the rhombic dodecahedron  (i.e., the associated cycles) may be chosen in such a manner that the two terms in the product of all cross-ratios associated with any edge cancel each other (cf.\ Figure \ref{dodecahedron}) so that
\bela{E51}
  \prod_{k=1}^{12} {\mathsf{cr}}_k = 1.
\ela
Accordingly, we conclude that if 11 cross-ratios are real then so is the remaining cross-ratio. Since four points are concyclic if and only if the associated cross-ratio is real, this concludes the proof.
\end{proof}

\subsubsection{The evolution of one-dimensional Cauchy data}

We have demonstrated that a fundamental point-circle complex $(\psf,\csf)$ is uniquely determined by the points corresponding to the two-dimensional triangular slice $A_2^0$ of the $\Z^3$ lattice. However, these points are constrained by the condition that triples of circles associated with the triangles $\nabla$ and $\Delta$ respectively be concurrent. It is now readily seen that admissible points $\psf(A_2^0)$ are uniquely constructed by prescribing points on a star-shaped region which divides the triangular lattice into three sectors of 120 degrees as displayed in Figure \ref{star} (left).
\begin{figure}
\centerline{\includegraphics[scale=0.3]{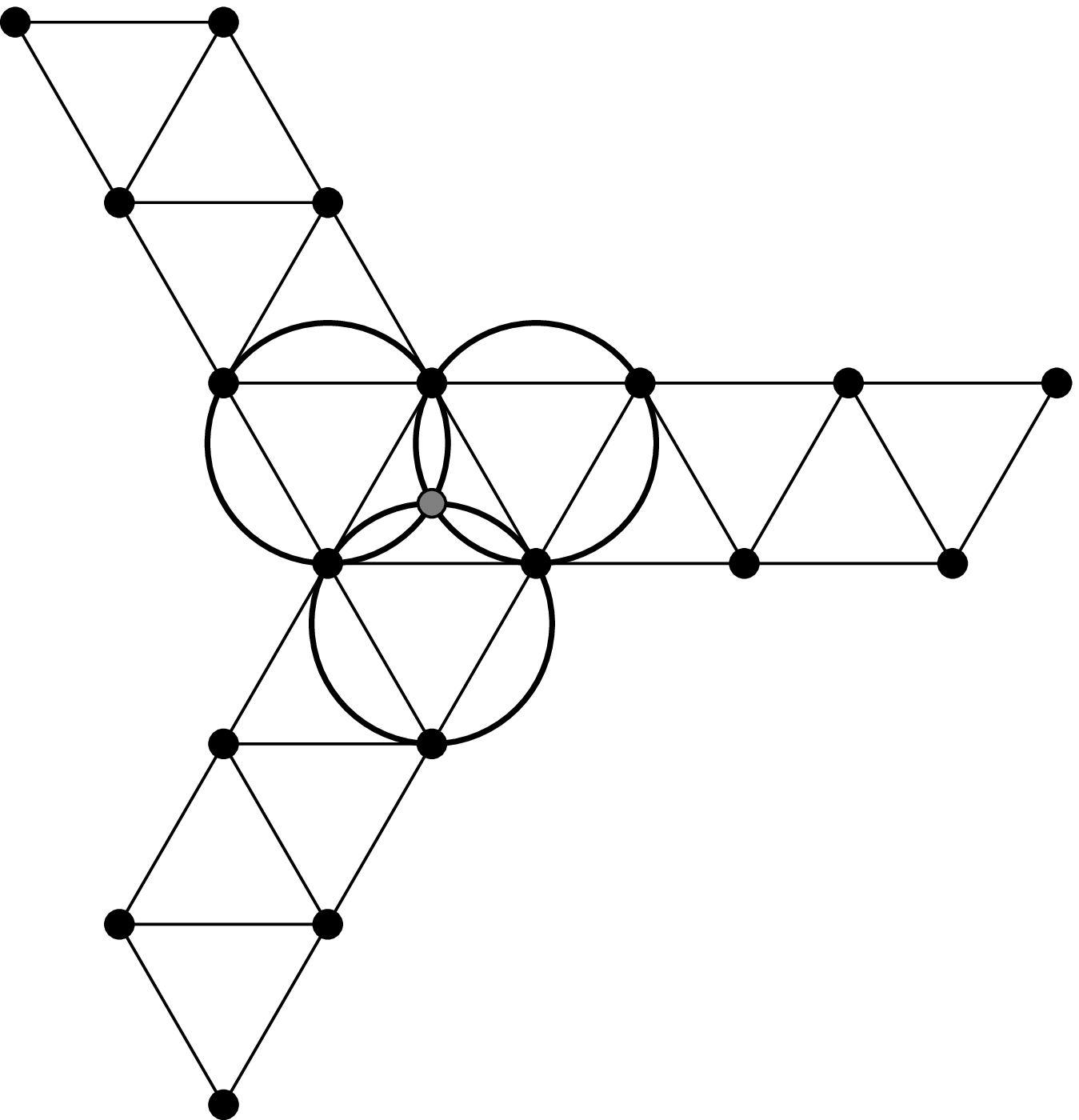}\includegraphics[scale=0.3]{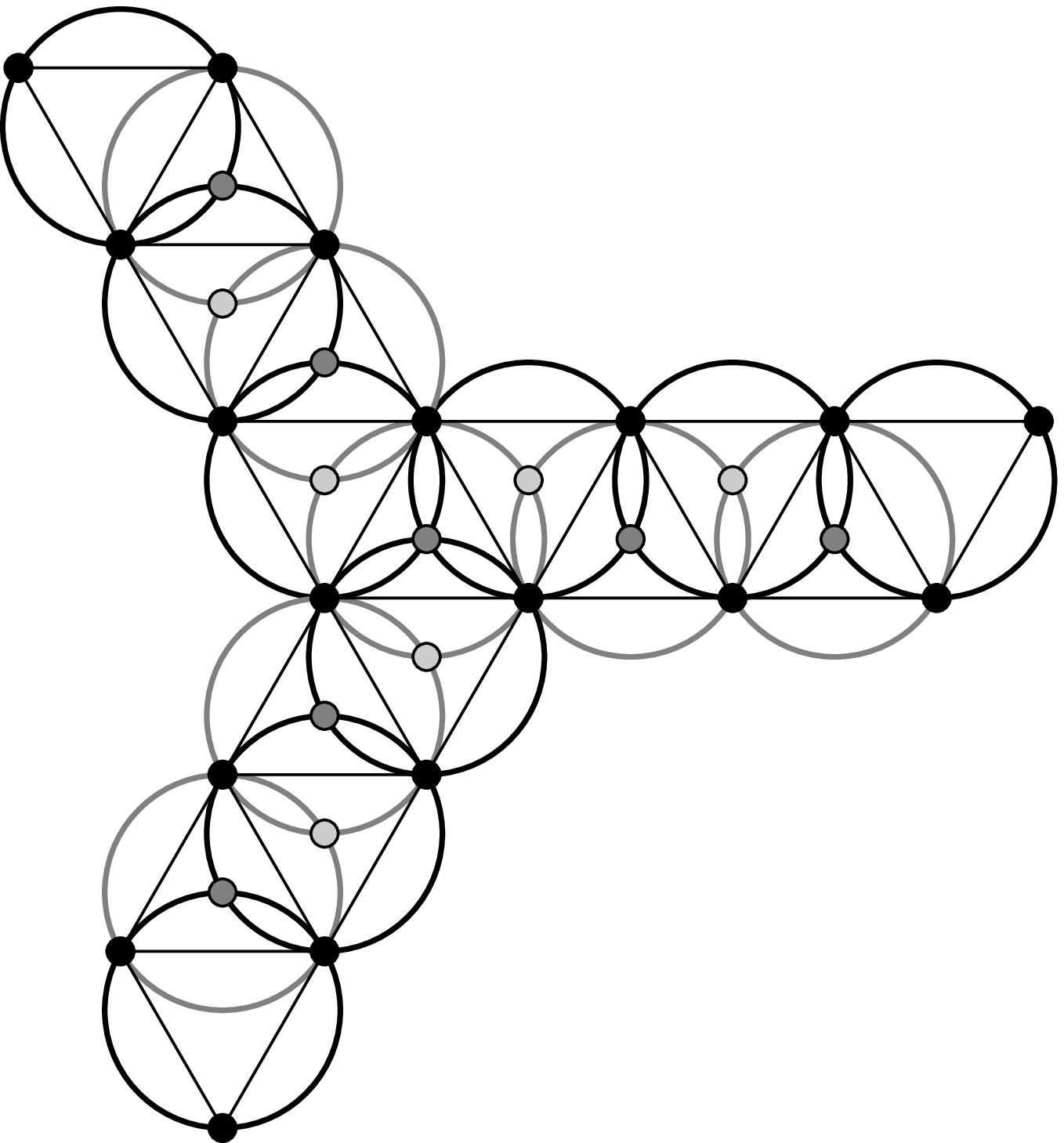}\includegraphics[scale=0.3]{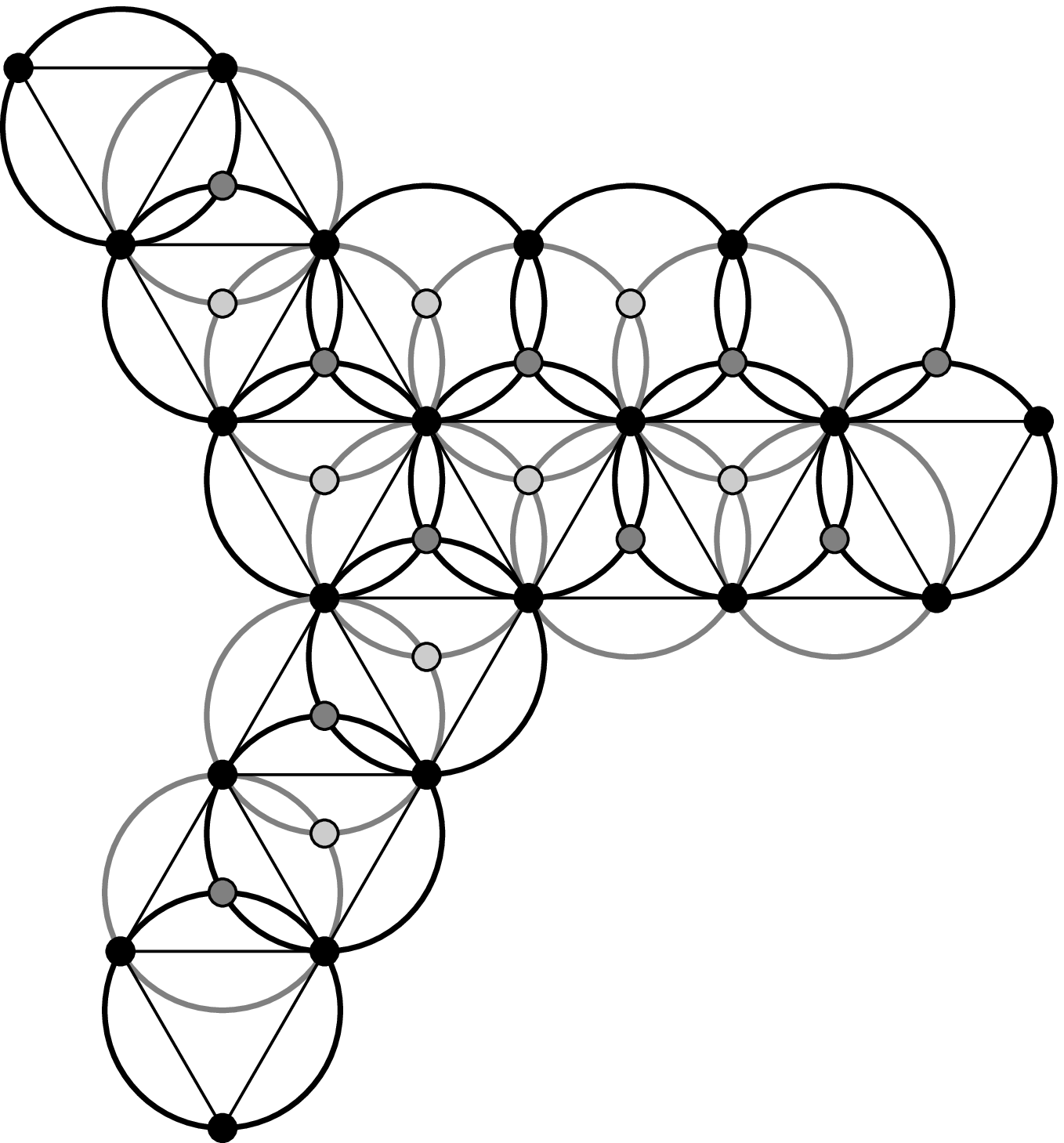}}
\caption{The construction of constrained Cauchy data for fundamental point-circle complexes.}
\label{star}
\end{figure}
The only constraint on the points is the concurrency of the three ``central'' (black) circles in $\csf(A_2^1)$, thereby defining the ``central'' (grey) point in $\psf(A_2^2)$. We begin by adding (black and grey) circles in $\csf(A_2^1)$ and $\csf(A_2^{-1})$ corresponding to the triangles $\nabla$ and $\Delta$ respectively as shown in Figure \ref{star} (middle). These determine, in turn, (grey and light grey) points in $\psf(A_2^2)$ and $\psf(A_2^{-2})$. We now observe that eight points of a rhombic configuration as depicted in Figure~\ref{rhombus} (left) and associated four circles and two points of intersection determine two additional circles and the remaining ninth point (Figure \ref{rhombus} (right)).
\begin{figure}
\centerline{\includegraphics[scale=0.5]{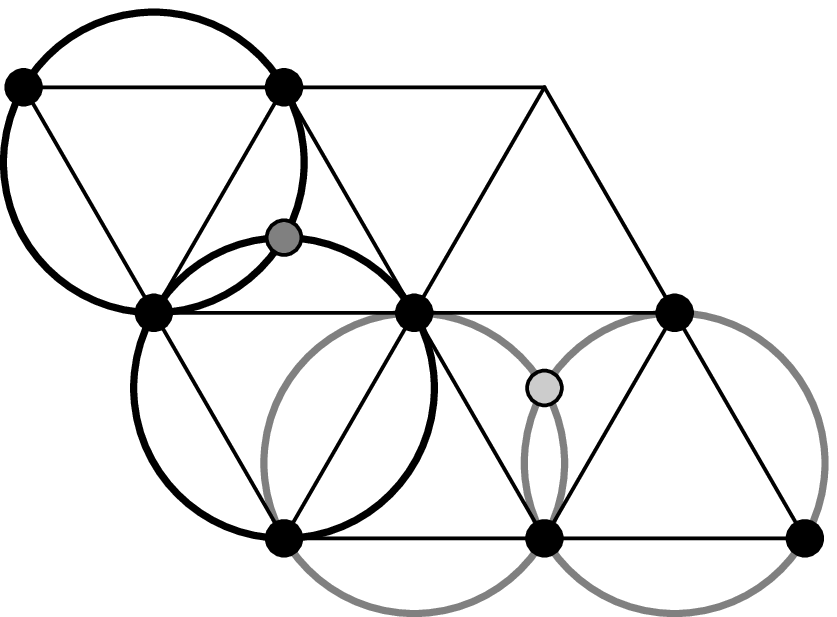}\qquad\qquad\includegraphics[scale=0.5]{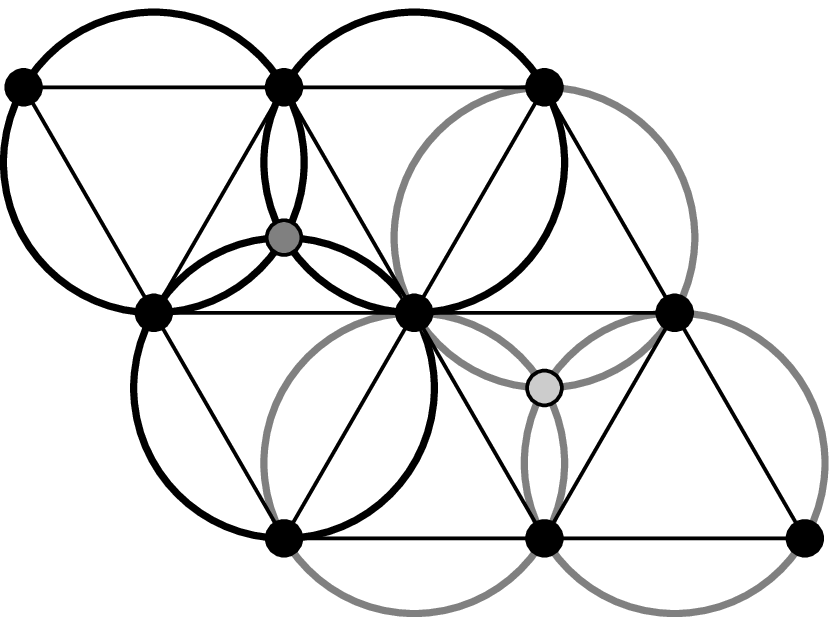}}
\caption{Eight (black) points (left) of a rhombic configuration uniquely determine a ninth (black) point (right).}
\label{rhombus}
\end{figure}
This procedure may now be applied to construct iteratively another ``horizontal'' row of (black) points in $\psf(A_2^0)$ together with the two associated rows of (grey and light grey) points in $\psf(A_2^2)$ and $\psf(A_2^{-2})$ as indicated in Figure \ref{star} (right) and, in fact, all points $\psf(A_2^0)$ in the corresponding sector. The other two sectors may be treated in an analogous manner. This may be summarised as follows.

\begin{theorem}
Fundamental point-circle complexes are determined by one-dimensional Cauchy data.
\end{theorem}

\begin{remark}
The above discussion suggests interpreting the constrained Cauchy data for fundamental point-circle complexes as two-dimensional configurations in their own right. As we have seen, this leads to the consideration of triangular lattices of points which are constrained by the concurrency of triples of circles associated with the same type of triangles. Equivalently, one could consider hexagonal circle packings which are constrained by the requirement that triples of ``missing'' circles be concurrent. In Figure \ref{cauchy_circle}, the hexagonal circle pattern consists of the black circles and black and grey points, while the ``missing'' circles are light grey and dashed. We have demonstrated that the concurrency of the grey circles implies the concurrency of the dashed circles. Hence, a third interpretation of the constrained Cauchy data is given by pairs of hexagonal circle patterns (black and grey circles) which are such that one circle pattern is constrained by the existence of the other. By construction, these circle patterns are geometrically and algebraically integrable as shown below. In the following, we refer to the constrained Cauchy data as {\em doubly hexagonal circle patterns}. It is noted that other types of integrable hexagonal circle patterns have been discussed in detail in \cite{BobenkoHoffmannSuris02,BobenkoHoffmann03}.
\end{remark}

\subsection{Integrability of fundamental point-circle complexes}

In order to interpret a fundamental point-circle complex as a particular fundamental complex in Lie circle geometry, we have to assign consistently an orientation to the circles. To this end, we consider an elementary cube of a given fundamental point-circle complex as in Figure \ref{cubes} (left). If we regard the points $\psf$ and $\psf_{12},\psf_{23},\psf_{13}$ as Lie circles of ``zero radius'' and choose an arbitrary orientation of the circles $\csf_1,\csf_2,\csf_3$ then there exist two Lie circles which ``touch'' the three  Lie circles  $\psf_{12},\psf_{23},\psf_{13}$. In the M\"obius picture, these correspond to the circle $\csf_{123}$ endowed with the two possible orientations. Since one of the two Lie circles is correlated to the Lie circle $\psf$ in the sense of Definition \ref{correlation_def}, this determines the orientation of the circle $\csf_{123}$. This leads to the following theorem.

\begin{theorem}\label{fund_point_circle}
Any fundamental point-circle complex $(\psf,\csf)$ may be interpreted as a fundamental circle complex by prescribing the orientation of the circles associated with the one-dimensional Cauchy data depicted in Figure \ref{star} (middle). The orientation of the remaining circles is determined by correlation. 
\end{theorem}

\begin{proof}
The orientation of the circles in Figure \ref{star} (middle) may be chosen arbitrarily except for the ``central'' (grey) circle in $\csf(A_2^{-1})$ which is required to be correlated to the ``central'' (grey) point in $\psf(A_2^{2})$. The orientation of the circles $\csf(A_2^{-1}),\csf(A_2^1)$ and $\csf(A_2^3)$ is then uniquely determined by correlation. The following lemma implies that the Lie circles correlated to the Lie circles $\csf(A_2^1)$ are indeed the points $\psf(A_2^4)$. The orientation of the circles $\csf(A_2^5)$ is determined by correlation and the same argument may be used to conclude that the points $\psf(A_2^6)$ are correlated to the circles $\csf(A_2^3)$. Iteration of this process in both forward and backward directions gives rise to a unique orientation of all remaining circles. 
\end{proof}

In the following, we refer to the points and circles of an elementary cube of a fundamental point-circle complex as being correlated if opposite pairs of points and circles regarded as Lie circles are correlated.

\begin{lemma}
If the points and circles of 7 elementary cubes of 8 adjacent elementary cubes of a fundamental point-circle complex  consisting of 14 points and 13 circles (cf.\ Figure \ref{correlation_circle}) are correlated then the points and circles of the remaining eighth elementary cube are likewise correlated.
\end{lemma}
\begin{figure}
\centerline{\includegraphics[scale=0.5]{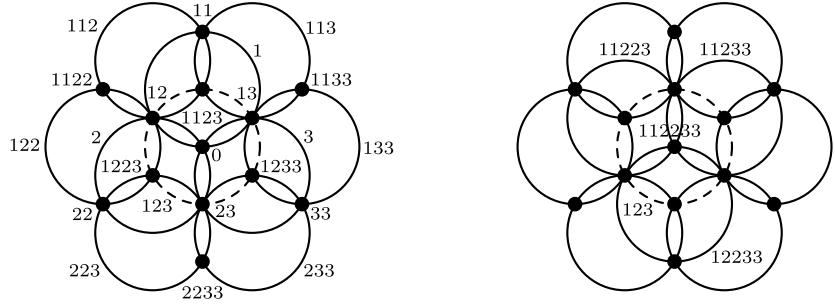}}
\caption{Points and circles of 8 adjacent elementary cubes of a fundamental point-circle complex. If the points and circles of 7 elementary cubes are correlated then the points and circles of the remaining eighth elementary cube are likewise correlated.}
\label{correlation_circle}
\end{figure}

\begin{proof}
Given a point $\psf$ in the plane with complex coordinate $z$, we may arbitrarily prescribe the centres $z_1,z_2,z_3$ of circles $\csf_!,\csf_2,\csf_3$. The requirement that these circles pass through $\psf$ determines the squares $r_1^2,r_2^2,r_3^2$ of their radii. The signs of the radii correspond to their orientation and may be chosen arbitrarily. The additional points of intersection of the pairs of circles $\csf_k,\csf_l$ are denoted by $\psf_{kl}$ as indicated in Figure \ref{correlation_circle}, wherein only the indices are shown and the point $\psf$ is labelled by $0$. These three points determine a circle $\csf_{123}$ up to its orientation. The latter is selected by the assumption that $\csf_{123}$ is correlated to $\psf$. If one considers the Lie geometric representation of the points and circles of this elementary cube of a point-circle complex as set down in Section 6 then, using the planarity property, one may explicitly determine the corresponding lines $\lsf,\ldots,\lsf_{123}$ in $\R\P^3$. Specifically, it may be directly verified by computer algebra that the radius $r_{123}$ is given by
\bela{Z1}
  r_{123} = r_1r_2r_3\frac{\left|\bear{ccc}1&z_1&\bar{z}_1\\ 1&z_2&\bar{z}_2\\ 1&z_3&\bar{z}_3\ear\right|}{
\left|\bear{cccc}1&z&\bar{z}&|z|^2\\ 1&z_1&\bar{z}_1&|z_1|^2\\ 1&z_2&\bar{z}_2&|z_2|^2\\  1&z_3&\bar{z}_3&|z_3|^2\ear\right|}.
\ela
The interested reader is referred to the corresponding Maple code which is available on the authors' websites. Now, circles $\csf_{113},\csf_{122},\csf_{233}$ are obtained by specifying their centres $z_{113},z_{122},z_{233}$. Once again, the signs of the radii $r_{113},r_{122},r_{233}$ may be prescribed. The additional points of intersection of the circle $\csf_{113}$ and the circles $\csf_1,\csf_{123}$ are denoted by $\psf_{11},\psf_{1123}$ respectively (cf.\ Figure \ref{correlation_circle}). Points $\psf_{22},\psf_{1223}$ and $\psf_{33},\psf_{1233}$ are defined in an analogous manner. The points $\psf_{11},\psf_{12},\psf_{1123}$ give rise to a circle $\csf_{112}$, the radius $r_{112}$ of which is given by the analogue of relation (\ref{Z1}). Similarly, the radii $r_{223}$ and $r_{133}$ of circles $\csf_{223}$ and $\csf_{133}$ are obtained. Accordingly, the correlated points and circles of four elementary cubes of a fundamental point-circle complex are known. Moreover, if $\psf_{1122},\psf_{2233},\psf_{1133}$ denote the additional points of intersection of the pairs of circles $(\csf_{112},\csf_{122}),(\csf_{223},\csf_{233}),(\csf_{133},\csf_{113})$ respectively then the radii $r_{11223},r_{12233},r_{11233}$ of circles $\csf_{11223},\csf_{12233},\csf_{11233}$ are given by the respective relations of the type (\ref{Z1}) and another three elementary cubes of correlated points and circles have been constructed. Finally, Theorem \ref{rhombic_theorem} implies that the circles $\csf_{11223},\csf_{12233},\csf_{11233}$ intersect in a point $\psf_{112233}$. However, {\it a priori}, it is not evident that the orientation of the circle $\csf_{123}'$ correlated to the point $\psf_{112233}$ coincides with the orientation of the circle $\csf_{123}$. In order to show this, we observe that the key relation (\ref{Z1}) may be formulated as
\bela{Z1a}
  r_{123}r_1r_2r_3 = f(z,z_1,z_2,z_3,\mbox{c.c.})
\ela
since the squares $r_i^2$ are known functions of the indicated arguments of $f$. Hence, the relations between the radii associated with the eight elementary cubes are of the form
\bela{Z1b}
 \bear{rlrl}
  r_{123}r_{1}r_{2}r_{3} = & f,\quad   &r_{123}r_{112}r_{113}r_{1} = & g_1\as
  r_{11223}r_{112}r_{122}r_{123} = & f_{12},\quad   &r_{223}r_{122}r_{123}r_{2} = & g_2\as
  r_{12233}r_{123}r_{223}r_{233} = & f_{23},\quad   &r_{233}r_{123}r_{133}r_{3} = & g_3\as
  r_{11233}r_{113}r_{123}r_{133} = & f_{13},\quad   &r_{12233}r_{11223}r_{11233}r_{123}' = & g_{123}.
\ear
\ela
Since the arguments of the right-hand sides of the above may be expressed in terms of the variables $z,z_1,z_2,z_3,z_{113},z_{122},z_{233}$ and their complex conjugates, the ratio of the product of (\ref{Z1b})$_{2,4,6,8}$ and the product of (\ref{Z1b})$_{1,3,5,7}$ reduces to
\bela{Z2}
  \frac{r'_{123}}{r_{123}} = F(z,z_1,z_2,z_3,z_{113},z_{122},z_{233},\mbox{c.c.})
\ela
which is a rational function of the indicated Cauchy data. On the other hand, the range of $F$ is $\{-1,1\}$ so that $F \equiv 1$ or $F \equiv -1$.  Consequently, it is sufficient to choose any convenient set of Cauchy data to determine the value of $F$. Indeed, a regular configuration of points and circles for which all unsigned radii of the circles are 1 leads, for instance, to $r_{123}=r_1r_2r_3$ by virtue of $|z_i-z|^2=1$ so that $f=\cdots = g_{123}=1$. Accordingly, $F\equiv 1$ which concludes the proof.
\end{proof}

Even though it is evident that fundamental point-circle complexes are described in terms of particular solutions of the symmetric $M$-system, it is non-trivial to derive the corresponding reduction of the $M$-system. An alternative approach to the algebraic description of  fundamental point-circle complexes is to analyse both geometrically and algebraically the underlying doubly hexagonal circle patterns. Here, we briefly derive the corresponding discrete integrable system. Thus, by construction, these patterns exist if and only if the points $\psf(A_2^0)$ are chosen in such a manner that the black and grey circles associated with the triangles $\nabla$ and $\Delta$ respectively of the triangular configurations displayed in Figure \ref{multiratio} are concurrent.
\begin{figure}
\centerline{\includegraphics[scale=0.5]{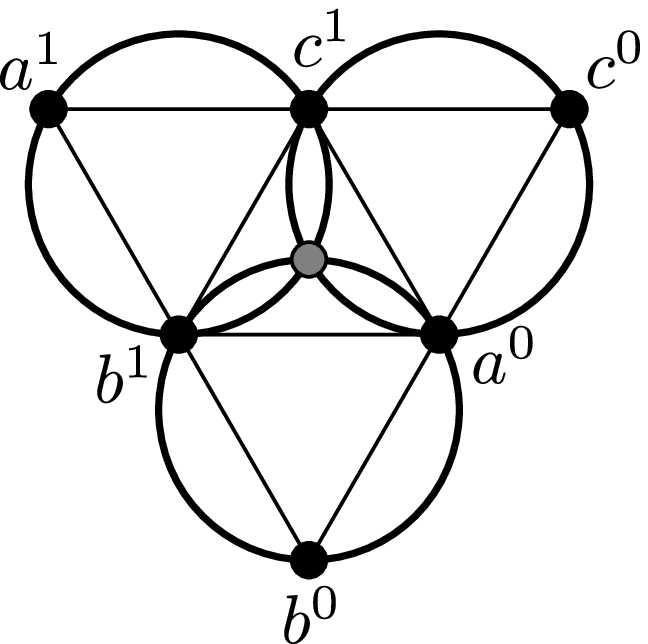}\qquad\qquad\includegraphics[scale=0.5]{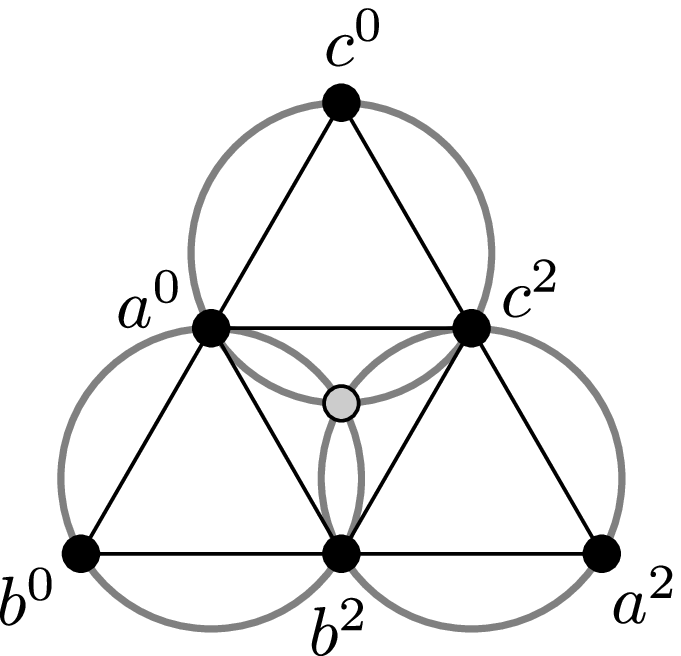}}
\caption{The building blocks of the doubly hexagonal circle patterns corresponding to the constrained Cauchy data of fundamental point-circle complexes.}
\label{multiratio}
\end{figure}
In algebraic terms, the concurrency of these triples of circles may be formulated in the following manner.

\begin{theorem}\label{alastair1}
The three circles passing through triples of points $(a^0,b^0,b^1)$, $(b^1,a^1,c^1)$ and $(c^1,c^0,a^0)$ belonging to a hexahedron $[a^0,b^0,b^1,a^1,c^1,c^0]$ in the plane (cf.\ Figure \ref{ninepoint} (left)) meet in a point $p$ if and only if the multi-ratio of complex numbers
\bela{E51a}
  \Msf(a^0,b^0,b^1,a^1,c^1,c^0) = \frac{(a^0-b^0)(b^1-a^1)(c^1-c^0)}{(b^0-b^1)(a^1-c^1)(c^0-a^0)}
\ela
is real.
\end{theorem}

\begin{figure}
\centerline{\includegraphics[scale=0.5]{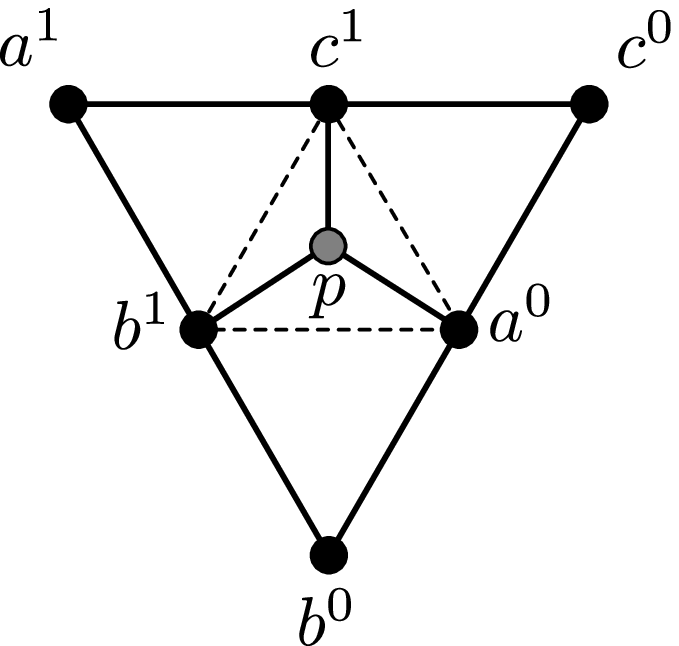}\qquad\qquad\includegraphics[scale=0.5]{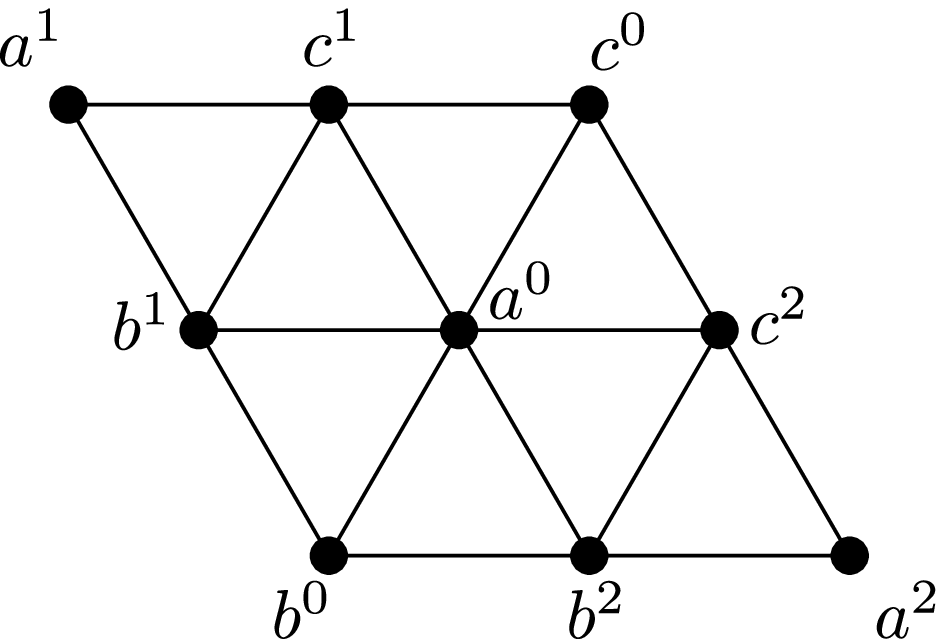}}
\caption{The labelling of the points of (half) a rhombic configuration.}
\label{ninepoint}
\end{figure}

\begin{proof}
Once again, we regard Figure \ref{ninepoint} (left) as a (combinatorial) closed polyhedral surface with a hexagonal face $[a^0,b^0,b^1,a^1,c^1,c^0]$ and three quadrilateral faces $[p,a^0,b^0,b^1]$, $[p,b^1,a^1,c^1]$, $[p,c^1,c^0,a^0]$ so that
\bela{E51b}
\mathsf{M}(a^0,b^0,b^1,a^1,c^1,c^0)\mathsf{cr}(p,a^0,b^0,b^1)\mathsf{cr}(p,b^1,a^1,c^1)\mathsf{cr}(p,c^1,c^0,a^0) = -1.
\ela
Hence, if $p$ is defined as the intersection of the circles passing through the triples of points $(a^0,b^0,b^1)$ and $(b^1,a^1,c^1)$ then the cross-ratio $\mathsf{cr}(p,c^1,c^0,a^0)$ is real if and only if the multi-ratio $\mathsf{M}(a^0,b^0,b^1,a^1,c^1,c^0)$ is real.
\end{proof}

\begin{remark}
The above theorem serves as a proof of the classical Miquel theorem discussed in Section 6.3. Indeed, if the circles   
passing through the triples of points $(a^0,b^0,b^1)$, $(b^1,a^1,c^1)$ and $(c^1,c^0,a^0)$ are concurrent then the above multi-ratio is real. However, a simple permutation of the arguments of the latter leads to its reciprocal and hence $\Msf(b^0,b^1,a^1,c^1,c^0,a^0)$ is real so that the three circles passing through the triples of points $(b^0,b^1,a^1)$, $(a^1,c^1,c^0)$ and $(c^0,a^0,b^0)$ are likewise concurrent.
\end{remark}

Doubly hexagonal circle patterns are now characterised in the following manner.

\begin{corollary}
The points $p(A_2^0)$ define doubly hexagonal circle patterns if and only if the multi-ratios
\bela{E51bc}
 \Msf(a^0,b^0,b^1,a^1,c^1,c^0),\quad \Msf(a^0,b^0,b^2,a^2,c^2,c^0)
\ela
associated with triangular configurations of the two types displayed in Figure \ref{multiratio} are real.
\end{corollary}

If we now focus on the nine points of a rhombic configuration of the type encountered in the preceding and displayed in Figure \ref{ninepoint} (right) then the two associated multi-ratio conditions read 
\bela{E51c}
 \bear{rl}
  \Msf_{(1)} = & \Msf(a^0,b^0,b^1,a^1,c^1,c^0)\in\R\as 
  \Msf_{(2)} = & \Msf(a^0,b^0,b^2,a^2,c^2,c^0)\in\R.
 \ear
\ela
These may be regarded as two conditions which uniquely determine the point $c^0$ in terms of the remaining eight points. Thus, the geometric evolution of the one-dimensional Cauchy data depicted in Figure \ref{rhombus} is algebraically encoded in the pair (\ref{E51c}). In order to obtain a compact form of this algebraic evolution, we first observe that
\bela{E94}
   \frac{\Msf_{(1)}}{\Msf^*_{(1)}} = \mathsf{cr}(c^1,c^0,a^0,c^2),\quad \frac{\Msf_{(2)}}{\Msf^*_{(2)}} = \mathsf{cr}(c^2,c^0,a^0,c^1),
\ela
wherein the multi-ratios $\Msf^*_{(1)}$ and $\Msf^*_{(2)}$ are defined by 
\bela{E53}
 \Msf^*_{(1)} = \Msf(a^0,b^0,b^1,a^1,c^1,c^2),\quad \Msf^*_{(2)} = M(a^0,b^0,b^2,a^2,c^2,c^1).
\ela
The cross-ratio relation 
\bela{E95}
   \mathsf{cr}(c^1,c^0,a^0,c^2) +  \mathsf{cr}(c^2,c^0,a^0,c^1) = 1
\ela
therefore shows that the four multi-ratios $\Msf_{(1)},\Msf_{(2)}$ and $\Msf^*_{(1)},\Msf^*_{(2)}$ obey the algebraic identity
\bela{E51d}
  \frac{\Msf_{(1)}}{\Msf^*_{(1)}} + \frac{\Msf_{(2)}}{\Msf^*_{(2)}}  = 1.
\ela
Finally, if we take the complex conjugate of the above relation and take into account that $\Msf_{(1)}$ and $\Msf_{(2)}$ are real for doubly hexagonal circle patterns, we arrive at the nine-point relation
\bela{E52}
  \frac{\Msf_{(1)}}{\bar{\Msf}^*_{(1)}} + \frac{\Msf_{(2)}}{\bar{\Msf}^*_{(2)}}  = 1.
\ela
The latter constitutes a well-defined evolution equation in the sense that it  represents a linear equation for the point $c^0$ with coefficients depending on the other eight points. Accordingly, conversely, if this relation holds then the two multi-ratios in (\ref{E51c}) are indeed real. Hence, the following theorem has been proven.

\begin{theorem}\label{alastair2}
Doubly hexagonal circle patterns and, in turn, fundamental point-circle complexes are governed by the integrable nine-point equation (\ref{E52}).
\end{theorem}

\begin{remark}
The above nine-point relation regarded as a lattice equation on an $A_2$ lattice implies the other two nine-point relations which are obtained by rotating the rhombic configuration by 120 degrees. The Cauchy problem in Figure \ref{star} (left) is applicable if, in each sector, an appropriate choice of one of the three nine-point relations is made. Hence, the Cauchy data for the algebraic description of doubly hexagonal circle patterns are one-dimensional, corresponding to a ``two-dimensional'' reduction of the symmetric $M$-system.
\end{remark}

A geometric proof of the integrability of fundamental point-circle complexes in the sense of  ``consistency around the (hyper)cube'' (see \cite{BobenkoSuris09} and references therein) is based on the existence of Clifford's classical $\mathcal{C}_4$ point-circle configuration \cite{Clifford71}. Thus, given an elementary cube of a fundamental point-circle complex such as the small cube in Figure \ref{hypercube}, we can extend this cube to a hypercube in the following manner. We prescribe another (black) circle passing through one of the points of the cube. 
\begin{figure}
\centerline{\includegraphics[scale=0.25]{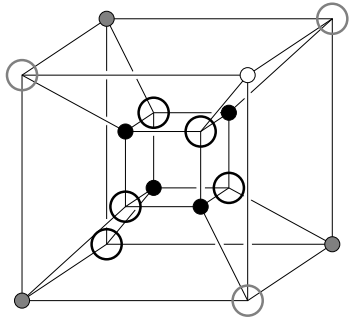}}
\caption{A hypercube extension of a cube of a fundamental point-circle complex.}
\label{hypercube}
\end{figure}
The four circles passing through this point have six points of intersection. Three of them are part of the cube and the other three points are coloured grey in Figure~\ref{hypercube}. The points of intersection of triples of the four circles generate another four circles, one of which is part of the original cube. Clifford's theorem then states that this circle and the other three (grey) circles meet in a (white) point. It is now easily verified that, by virtue of Clifford's theorem, the hypercube extensions of all remaining elementary cubes of the fundamental point-circle complex exist and are unique. In this manner, one obtains an extension of the fundamental point-circle complex defined on the union of two ``parallel'' $\Z^3$ slices of a $\Z^4$ lattice. The fundamental point-circle complex associated with the $\Z^3$ lattice parallel to the original $\Z^3$ lattice constitutes a B\"acklund transform of the original fundamental point-circle complex in the sense of integrable systems theory \cite{BobenkoSuris09,RogersSchief02}. Furthermore, this two-layer point-circle configuration may be extended on the entire $\Z^4$ lattice.

\begin{remark}
By construction, an elementary cube of a fundamental circle complex may be regarded as a generalisation of an elementary cube of a fundamental point-circle complex which, in turn, is nothing but a $\mathcal{C}_3$ point-circle configuration. As demonstrated above, the latter Clifford configuration may be extended to a (non-unique) $\mathcal{C}_4$ configuration and, in fact, to a $\mathcal{C}_n$ configuration for any $n>3$. Since the symmetric $M$-system may be imposed on a $\Z^n$ lattice for any $n\geq3$ and, in particular, on any $n$-dimensional hypercube, elementary cubes of fundamental circle complexes extended to $n$-dimensional hypercubes generalise classical $\mathcal{C}_n$ point-circle configurations. Hence, we have constructed a canonical Lie geometric generalisation of Clifford's classical chain of circle theorems \cite{LonguetHiggins72}. It is emphasised that the requirement of correlated opposite circles is crucial in this context since it guarantees the existence of these hypercube extensions.
\end{remark}

\section{Hyperbolic and Laguerre geometry}

\subsection{A parametrisation of the Lie quadric}

In the previous section, we have exploited the interpretation of M\"obius geometry as a sub-geometry of Lie circle geometry obtained by considering the intersection of the Lie quadric $Q^3$ and a hyperplane of signature $(3,1)$. Indeed, the M\"obius quadric $Q^2$ representing circles of zero radius (points) on the plane is the orthogonal complement of a ``time-like'' vector $\Esf$ so that
\bela{E55}
  Q^2 = \{\Vsf\in Q^3 : \langle \Vsf,\Esf\rangle = 0\},\quad \langle\Esf,\Esf\rangle = -1.
\ela
If, for convenience, we choose
\bela{E56}
  \diag\left[1,1,-\frac{1}{2}\left(\bear{cc}0&1\\ 1&0\ear\right),-1\right]
\ela
as the metric of the space $\R^{3,2}$ of homogeneous coordinates then we may set
\bela{E57}
  \Esf = (0,0,0,0,1)
\ela
without loss of generality. The geometric interpretation of the following parametrisation of the points of the Lie quadric reflects this choice.
\medskip

\noindent
Oriented circles of signed radius $r\neq0$ centred at $(x,y)$:
\bela{E58}  
  \Vsf_\ocircle = (x,y,1,x^2+y^2-r^2,r).
\ela
Oriented lines $\{(x,y) : xv + yw = d\}$ with $v^2 + w^2 = 1$:
\bela{E59}
 \Vsf_{|} = (v,w,0,2d,1).
\ela
Points $(x,y)$:
\bela{E60}  
  \Vsf_\bullet = (x,y,1,x^2+y^2,0).
\ela
Point at ``infinity'':
\bela{E61}  
  \Vsf_\infty = (0,0,0,1,0).
\ela
\medskip

\noindent
In particular, it is seen that $Q^2$ encodes all points in the plane and the point at infinity, that is, in short-hand notation, $Q^2 = \{\Vsf_\bullet,\Vsf_\infty\}$.

\subsection{The geometry of real cross-ratios}

The proof of the existence and algebraic integrability of fundamental point-circle complexes has been based entirely on combinatorial arguments and properties of the cross-ratio of four points. This also applies to Clifford's theorem which is a consequence of Miquel's theorem. The latter, regarded as a theorem for hexahedral point-circle configurations, may be proven in the same manner as Theorem \ref{rhombic_theorem} associated with rhombic dodecahedral point-circle configurations. This observation will be exploited in Sections 6.3 and 6.4 to formulate and prove analogues of relevant theorems such as Miquel's theorem in hyperbolic and Laguerre geometry. To this end, it turns out convenient to reformulate the condition for four points to lie on a(n) (oriented) circle. Thus, a point
represented by 
\bela{E62}
  \Vsf_{(k)} = (x_{(k)},y_{(k)},1,x_{(k)}^2+y_{(k)}^2,0)
\ela
lies on a circle represented by $\Vsf_\ocircle$ if and only if $\langle\Vsf_{(k)},\Vsf_\ocircle\rangle = 0$, that is,
\bela{E63}
 \left(\bear{ccc}x_{(k)}& y_{(k)}&\dis -\frac{1}{2}\ear\right)\left(\bear{c}x\\y\\ x^2+y^2-r^2\ear\right) = \frac{1}{2}(x_{(k)}^2 + y_{(k)}^2).
\ela
Accordingly, four points represented by $\Vsf_{(k)}$, $k=1,\ldots,4$ lie on a circle (or a line) if and only if the above linear system of equations has vanishing associated determinant, that is,
\bela{E64}
  \left|\bear{cccc} 1 & x_{(1)} & y_{(1)} & x_{(1)}^2 + y_{(1)}^2\as
                           1 & x_{(2)} & y_{(2)} & x_{(2)}^2 + y_{(2)}^2\as
                           1 & x_{(3)} & y_{(3)} & x_{(3)}^2 + y_{(3)}^2\as
                           1 & x_{(4)} & y_{(4)} & x_{(4)}^2 + y_{(4)}^2
  \ear\right| = 0.
\ela
It is noted that the solution of the linear system (\ref{E63}) is such that the circle (of either orientation) is real. In terms of complex numbers $z = x + iy$, the above condition may be expressed as
\bela{E65}
  \left|\bear{cccc} 1 & z_{(\cdot)} & \bar{z}_{(\cdot)} & |z_{(\cdot)}|^2\ear\right| = 0
\ela
in shorthand notation. On the other hand, it is readily verified that
\bela{E66}
   \left|\bear{cccc} 1 & z_{(\cdot)} & \bar{z}_{(\cdot)} & |z_{(\cdot)}|^2\ear\right| = 
  - (z_{(1)} - z_{(2)})(z_{(3)} - z_{(4)})(\bar{z}_{(2)} - \bar{z}_{(3)})(\bar{z}_{(4)} - \bar{z}_{(1)}) +\, \mbox{c.c.} 
\ela
so that
\bela{E67}
  \left|\bear{cccc} 1 & z_{(\cdot)} & \bar{z}_{(\cdot)} & |z_{(\cdot)}|^2\ear\right| = 0\quad\Leftrightarrow\quad
\frac{(z_{(1)} - z_{(2)})(z_{(3)} - z_{(4)})}{(z_{(2)} - z_{(3)})(z_{(4)} - z_{(1)})}\in\R.
\ela
Hence, we have retrieved the fact that four points lie on a circle (of, possibly, infinite radius) if and only if the cross-ratio of the four points is real. The algebraic equivalence (\ref{E67}) which holds, {\it mutatis mutandis}, in the case of ``generalised'' complex numbers turns out to be key in the following discussion.

\subsection{Hyperbolic geometry}

We now examine fundamental complexes associated with a ``space-like'' vector~$\Esf$. This corresponds to the intersection of the Lie quadric $Q^3$ with a hyperplane of signature $(2,2)$. If we make the choice
\bela{E68}
  \Esf = (0,1,0,0,0)
\ela
without loss of generality then the resulting quadric 
\bela{E68a}
  Q^2 = \{\Vsf\in Q^3 : \langle \Vsf,\Esf\rangle = 0\},\quad \langle\Esf,\Esf\rangle = 1
\ela
is composed of the points $Q^2 = \{\Vsf_\ominus,\Vsf_\perp,\Vsf_\infty\}$ with
\bela{E69}
 \bear{rl}
  \Vsf_\ominus = & (x,0,1,x^2 - r^2,r)\as
  \Vsf_\perp = & (\pm 1,0,0,2d,1),
  \ear
\ela
including the case $r=0$. Hence, $Q^2$ represents the set of oriented circles and straight lines on the plane which are orthogonal to the $x$-axis (and the point at infinity). Accordingly, the corresponding semi-circles and semi-lines in the upper half-plane may be identified with oriented geodesics of the Poincar\'e half-plane model of hyperbolic geometry \cite{Hyperbolic} so that, for lack of a better expression, we refer to the oriented circles and lines encoded in $Q^2$ as `geodesics'. It is observed that the Poincar\'e disk model is obtained by making the choice $\Esf = (0,0,1,-1,0)$ with the quadric $Q^2$ representing the circles and lines which are orthogonal to the unit circle. In this connection, it is emphasised that the geometry discussed in this section does not constitute standard hyperbolic geometry but a natural generalisation. Here, we are concerned with the geometry of oriented touching circles in the hyperbolic plane (or disk). The corresponding symmetry group consists of the Lie transformations of oriented circles which preserve $\mathsf{E}$. It includes but is not confined to hyperbolic isometries. Due to its analogy to Laguerre geometry, one may refer to the geometry considered here as ``hyperbolic Laguerre geometry''.

\subsubsection{The geometry of double numbers}

As in the M\"obius case, we now determine the constraint on four geodesic circles which guarantees that these are tangent to a common circle (or line). Since the condition that four geodesic circles represented by
\bela{E69a}
 \Vsf_{(k)} = (x_{(k)},0,1,x_{(k)}^2 - r_{(k)}^2,r_{(k)}),\quad k=1,\ldots,4 
\ela
touch a circle represented by $\Vsf_\ocircle$ is given by $\langle\Vsf_{(k)},\Vsf_\ocircle\rangle = 0$, we obtain the linear system
\bela{E70}
 \left(\bear{ccc}x_{(k)}& -r_{(k)}&\dis -\frac{1}{2}\ear\right)\left(\bear{c}x\\r\\ x^2+y^2-r^2\ear\right) = \frac{1}{2}(x_{(k)}^2 - r_{(k)}^2).
\ela
Accordingly, the four geodesic circles touch a circle (or a line) if and only if
\bela{E71}
   \left|\bear{cccc} 1 & x_{(\cdot)} & r_{(\cdot)} & x_{(\cdot)}^2 - r_{(\cdot)}^2\ear\right| = 0.
\ela
In terms of the associative-commutative algebra of {\em double numbers} \cite{Yaglom68}
\bela{E72}
  z = x + jr,\quad \bar{z} = x - j r,\quad |z|^2 = z\bar{z} = x^2 - r^2,\quad j^2 = 1,
\ela
this may be formulated, once again, as
\bela{E73}
  \left|\bear{cccc} 1 & z_{(\cdot)} & \bar{z}_{(\cdot)} & |z_{(\cdot)}|^2\ear\right| = 0.
\ela
It is evident that there exists a one-to-one correspondence between oriented geodesic circles and double numbers with the ``real'' and ``imaginary" parts representing the location and the signed radius respectively of a geodesic circle. Since the identity (\ref{E66}) is also valid for double numbers, four non-touching oriented geodesic circles labelled by the double numbers $z_{(k)}$, $k=1,\ldots,4$ have oriented contact with a circle (of, possibly, infinite radius) if and only if the cross-ratio
\bela{E74}
 \mathsf{cr}(z_{(1)},z_{(2)},z_{(3)},z_{(4)}) = \frac{(z_{(1)} - z_{(2)})(z_{(3)} - z_{(4)})}{(z_{(2)} - z_{(3)})(z_{(4)} - z_{(1)})}
\ela
is ``real'' \cite{Yaglom68}.  Here, the term `non-touching' refers to geodesic circles which do not have oriented contact. This guarantees that the above cross-ratio is well-defined since two geodesic circles labelled by $z_\ominus$ and $z_\oslash$ have oriented contact if and only if $|z_\ominus - z_\oslash|^2 = 0$ as may be verified by evaluating the corresponding condition $\langle\Vsf_\ominus,\Vsf_\oslash\rangle=0$.

The connection between double numbers and oriented geodesic circles may now be exploited to translate theorems in M\"obius geometry into theorems in hyperbolic geometry if these may be proven in a purely combinatorial manner using cross-ratios. For instance, Miquel's theorem \cite{Pedoe88} states that, given four circles $\csf^{1234},\ldots,\csf^{7812}$ in the plane which cyclically intersect in four pairs of points $(\psf^1,\psf^2),\ldots,(\psf^7,\psf^8)$, if four points of intersection $\psf^1,\psf^3,\psf^5,\psf^7$ are concyclic then the remaining four points $\psf^2,\psf^4,\psf^6,\psf^8$ are concyclic as indicated in Figure \ref{miquel}.
\begin{figure}
\centerline{\includegraphics[scale=0.4]{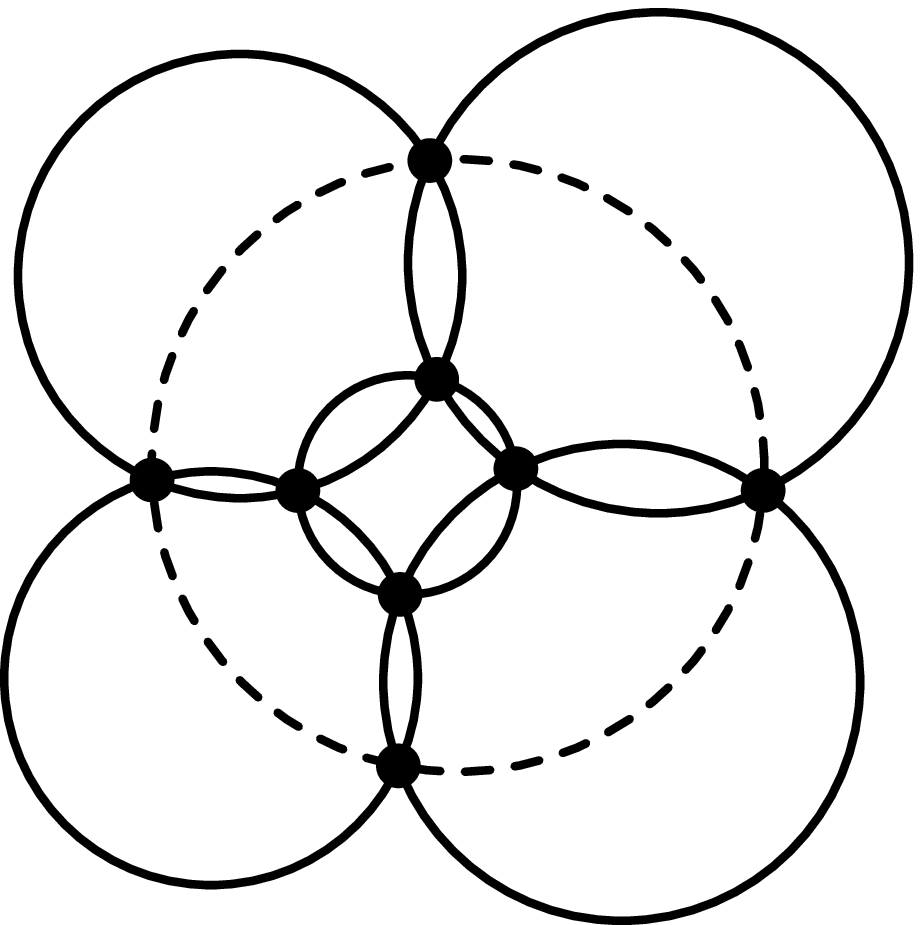}\qquad\qquad\includegraphics[scale=0.4]{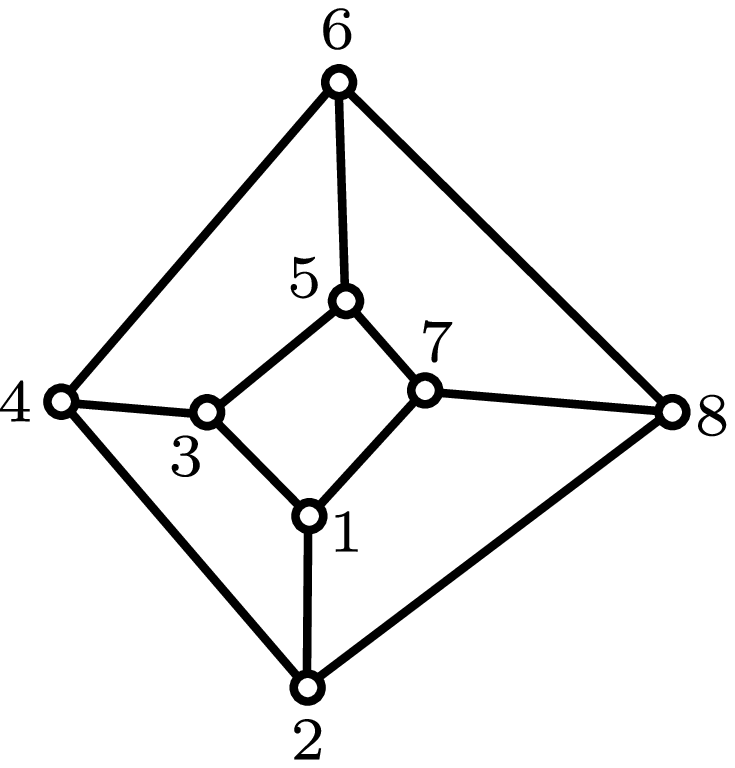}}
\caption{Miquel's theorem and the hexahedral combinatorics of the associated configuration.}
\label{miquel}
\end{figure}
This is readily proven by regarding the Miquel figure as a configuration of six circles and eight points which has the combinatorics of a cube as indicated in Figure \ref{miquel}. Indeed, as in the case of rhombic dodecahedral point-circle configurations, for any hexahedral configuration of points $\psf^1,\ldots,\psf^8$, one may define six cross-ratios $\mathsf{cr}_k$ associated with the faces of the configuration in such a manner that
\bela{E75}
  \prod_{k=1}^6\mathsf{cr}_k = 1.
\ela
Hence, if five faces are circular then the corresponding five cross-ratios are real and (\ref{E75}) implies that the sixth cross-ratio is likewise real. 

If we replace complex numbers by double numbers then, by associating geodesic circles and circles with the vertices and faces respectively of the combinatorial hexahedron in Figure \ref{miquel} (right), the above argument leads to the following geometric statement.

\begin{theorem}
Let $\gsf^1,\ldots,\gsf^8$ be a hexahedral configuration of geodesic circles. If there exist four circles $\csf^{1234},\ldots,\csf^{7812}$ which have oriented contact with the respective pairs of geodesic circles $(\gsf^1,\gsf^2),\ldots,(\gsf^7,\gsf^8)$ and there exists a fifth circle $\csf^{1357}$ which has oriented contact with the geodesic circles $\gsf^1,\gsf^3,\gsf^5,\gsf^7$ then the geodesic circles $\gsf^2,\gsf^4,\gsf^6,\gsf^8$ touch a circle $\csf^{2468}$ in an oriented manner.
\end{theorem}

\noindent
It is observed that the sixth circle $\csf^{2468}$ is not unique. There exist two such circles which are related by reflection in the $x$-axis. In fact, if we replace by its reflection any circle which is contained in the lower half-plane or touches the $x$-axis from below then we obtain a Miquel-type configuration of eight geodesics and six (hyper-/horo-)cycles in the sense of hyperbolic geometry \cite{Hyperbolic}. A typical Miquel-type configuration within the Poincar\'e disk model obtained by applying a conformal transformation is displayed in Figure \ref{miquelhyperbolic} (left). The eight geodesics are represented by oriented circular arcs which meet a given circle orthogonally.
\begin{figure}
\centerline{\includegraphics[scale=0.18]{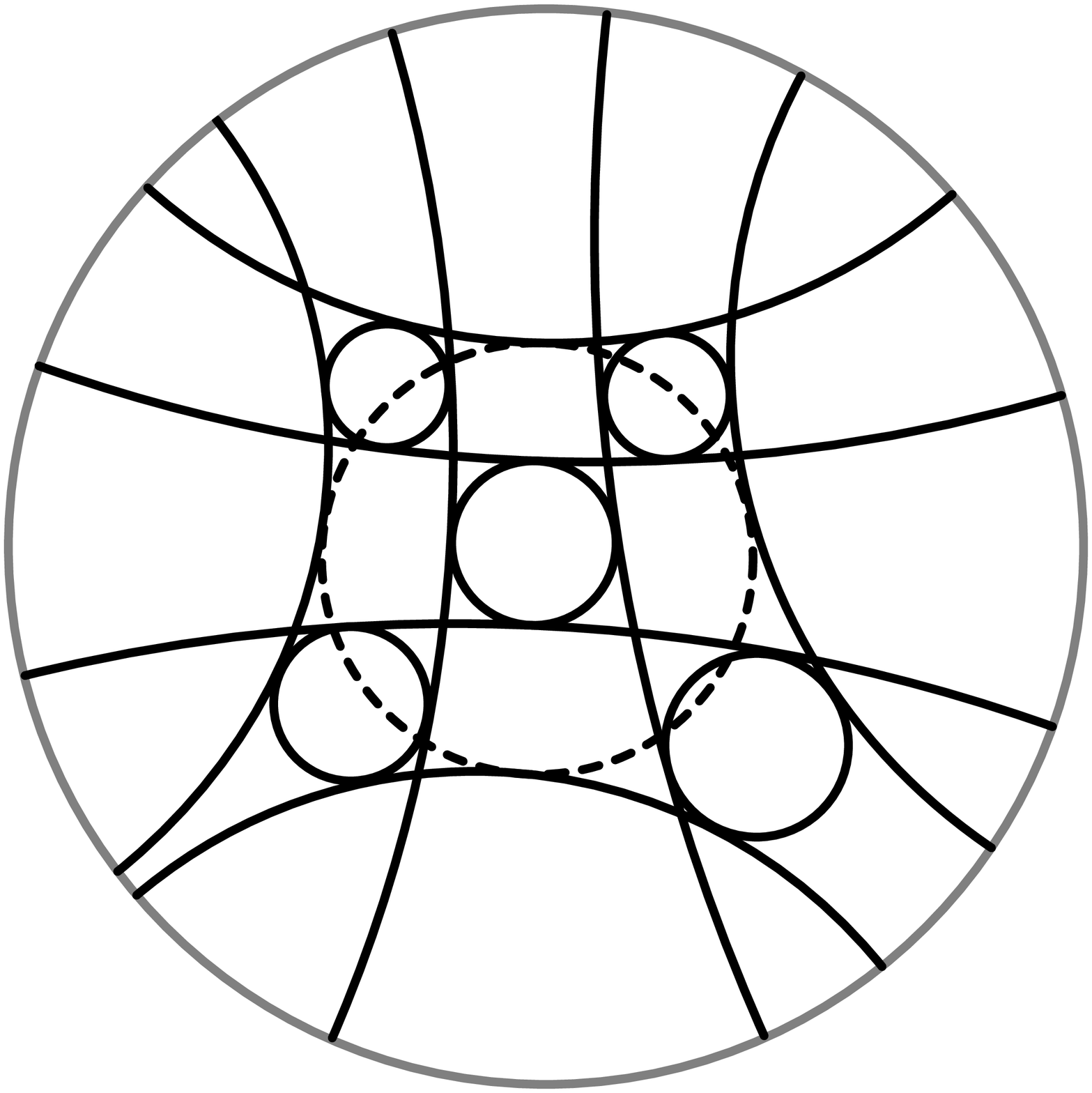}\qquad\qquad\includegraphics[scale=0.25]{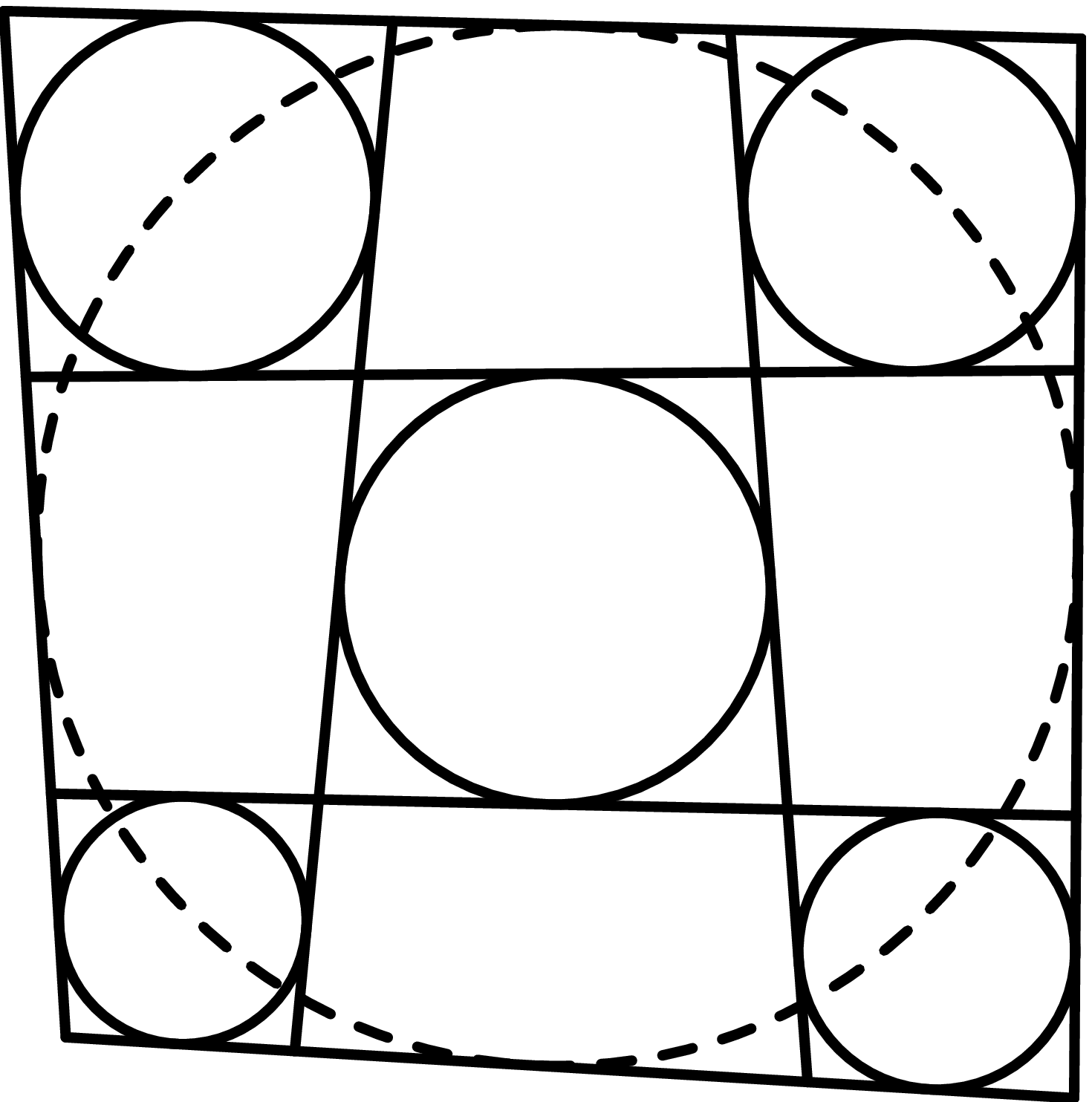}}
\caption{Analogues of Miquel figures in hyperbolic geometry (left) and Laguerre geometry (right).}
\label{miquelhyperbolic}
\end{figure}

\subsubsection{Fundamental geodesic-circle complexes}

In order to define a canonical analogue of fundamental point-circle complexes in hyperbolic geometry, we consider maps of the type
\bela{76}
 \bear{rl}
  \gsf : \De_3 & \rightarrow \{\mbox{oriented geodesic circles in $\R^2$}\}\as
  \csf : \Do_3 & \rightarrow \{\mbox{oriented circles in $\R^2$}\}
 \ear
\ela
which we combine to maps $(\gsf,\csf)$ defined on $\Z^3 = \De_3\cup\Do_3$. The following dictionary of incidences now allows us to replicate the geometric and algebraic discussion of the previous section within hyperbolic geometry.
\begin{center}
    \begin{tabular}{ | p{5.5cm} | p{5.5cm} |}
    \hline
    M\"obius geometry & hyperbolic geometry\\ \hline\hline
      a point on a circle & a geodesic circle and a circle in oriented contact\\ \hline
      a pair of points of intersection of two circles &  a (unique) pair of geodesic circles in oriented contact with two circles\\
    \hline
   a circle (either orientation) through three points & a (unique) circle (and its reflection) in oriented contact with three geodesic circles\\
   \hline
    \end{tabular}
\end{center}
\begin{definition}
A configuration of oriented geodesic circles and circles combinatorially attached to the vertices of the even and odd sublattices of a $\Z^3$ lattice respectively is termed a {\em fundamental geodesic-circle complex} if the geodesic circles and circles are incident along the edges of the $\Z^3$ lattice.
\end{definition}

Fundamental geodesic-circle complexes may be constructed in the same manner as fundamental point-circle complexes. The first step is to consider a Cauchy problem of the combinatorial type displayed in Figure \ref{star} (left) with the points being interpreted as geodesic circles and the three circles being in oriented contact with the relevant geodesic circles. The geodesic circles $\gsf(A_2^{-2}),\gsf(A_2^0),\gsf(A_2^2)$ and circles $\csf(A_2^{-1}),\csf(A_2^1)$ of a ``two-dimensional'' geodesic-point circle configuration may then be determined iteratively as in the M\"obius geometric case. The only difference is that the circles are defined up to reflection in the $x$-axis. However, this binary choice does not affect the geodesic circles. Secondly, the analogue of Theorem \ref{rhombic_theorem} for rhombic dodecahedral geodesic-circle configurations gives rise to a well-defined propagation of the constrained two-dimensional Cauchy data. In this manner, one obtains a complete fundamental geodesic-circle complex which is unique up to the above-mentioned binary choice of the circles. Thirdly, B\"acklund transforms of fundamental geodesic-circle complexes may be constructed by means of the analogue of Clifford's theorem in hyperbolic geometry in the sense of the above-mentioned dictionary. Its validity is due to the fact that Clifford's theorem may be regarded as a purely combinatorial application of Miquel's theorem. Hence, we have come to the following conclusion which is justified in the same manner as Theorem \ref{fund_point_circle}. Here, it is noted that the key relation (\ref{Z1}) is also valid in the current context if it is formulated in terms of real quantities $x,y,r$ and the formal substitution $(y,r)\rightarrow i(r,y)$ is made so that
\bela{76a}
  y_{123} = -y_1y_2y_3\frac{\left|\bear{ccc}1&z_1&\bar{z}_1\\ 1&z_2&\bar{z}_2\\ 1&z_3&\bar{z}_3\ear\right|}{
\left|\bear{cccc}1&z&\bar{z}&|z|^2\\ 1&z_1&\bar{z}_1&|z_1|^2\\ 1&z_2&\bar{z}_2&|z_2|^2\\  1&z_3&\bar{z}_3&|z_3|^2\ear\right|}
\ela
on use of double numbers $z = x + j r$.

\begin{theorem}
The geodesic circles of a fundamental geodesic-circle complex $(\gsf,\csf)$ are uniquely determined by one-dimensional Cauchy data of the combinatorial type depicted in Figure \ref{star} (left). The circles of a fundamental geodesic-circle complex come in pairs related by reflection in the $x$-axis. If, for a fixed choice of the circles associated with the one-dimensional Cauchy data, the choice of the remaining circles is dictated by correlation then a fundamental geodesic-circle complex may be interpreted as a particular fundamental circle complex. In this sense, fundamental geodesic-circle complexes are governed by a reduction of the symmetric $M$-system. 
\end{theorem}

\begin{remark}
As stated in the above theorem, for four geodesic circles and three circles which are incident along the edges of an elementary cube of a $\Z^3$ lattice, there exists a canonical choice for the fourth circle, namely the circle which is correlated to the ``opposite'' geodesic circle. However, the construction of fundamental geodesic-circle complexes within the Poincar\'e half-plane model of hyperbolic geometry may require reflecting appropriate circles in the $x$-axis. In the case of the Poincar\'e disk model,  one may either subsequently apply a conformal transformation or choose the space-like vector $\Esf=(0,0,1,-1,0)$ and invert circles in the unit circle whenever appropriate. It is also remarked that the fact that the binary choice of the circles does not affect the geodesic circles is reflected by (\ref{E52}), that is,
\bela{E77}
  \frac{\Msf_{(1)}}{\bar{\Msf}^*_{(1)}} + \frac{\Msf_{(2)}}{\bar{\Msf}^*_{(2)}}  = 1,
\ela
which, in the current context, constitutes an evolution equation for double numbers governing the geodesic circles only.
\end{remark}

\subsection{Laguerre geometry}

Laguerre geometry of the plane \cite{Blaschke29} is obtained from Lie circle geometry by considering a ``light-like'' vector $\Esf$, corresponding to the intersection of the Lie quadric $Q^3$ with a hyperplane of signature $(2,1,0)$. In this case, we may choose
\bela{E78}
  \Esf = (0,0,0,1,0)
\ela
without loss of generality
so that the quadric 
\bela{E79}
  Q^2 = \{\Vsf\in Q^3 : \langle\Vsf,\Esf\rangle = 0\},\quad \langle\Esf,\Esf\rangle = 0
\ela
consists of the points $Q^2 = \{\Vsf_|,\Vsf_\infty\}$. Hence, $Q^2$ represents the set of oriented lines (and the point at infinity). In the same manner as before, we can now translate incidences in M\"obius geometry into incidences in Laguerre geometry.
\begin{center}
    \begin{tabular}{ | p{5.5cm} | p{5.5cm} |}
    \hline
    M\"obius geometry & Laguerre geometry\\ \hline\hline
      a point on a circle & a line and a circle in oriented contact\\ \hline
      a pair of points of intersection of two circles &  a (unique) pair of lines in oriented contact with two circles\\
    \hline
   a circle (either orientation) through three points & a (unique) circle (and the point at infinity) in oriented contact with three lines\\
   \hline
    \end{tabular}
\end{center}
In this way, one retrieves, for instance, the analogue of Miquel's theorem in Laguerre geometry as stated below \cite{Yaglom68} (see Figure \ref{miquelhyperbolic} (right)).

\begin{theorem}
Let $\Lsf^1,\ldots,\Lsf^8$ be a hexahedral configuration of lines in the plane. If there exist four circles $\csf^{1234},\ldots,\csf^{7812}$ which have oriented contact with the respective pairs of lines $(\Lsf^1,\Lsf^2),\ldots,(\Lsf^7,\Lsf^8)$ and there exists a fifth circle $\csf^{1357}$ which has oriented contact with the lines $\Lsf^1,\Lsf^3,\Lsf^5,\Lsf^7$ then the lines $\Lsf^2,\Lsf^4,\Lsf^6,\Lsf^8$ touch a unique circle $\csf^{2468}$ in an oriented manner.
\end{theorem}

\subsubsection{The geometry of dual numbers}

Once again, the proof of the above variant of Miquel's theorem is of a purely combinatorial nature exploiting the geometry of cross-ratios of {\em dual numbers} \cite{Yaglom68}
\bela{E80}
  z = \dualu + j \dualv\quad \bar{z} = \dualu - j \dualv,\quad |z|^2 = z\bar{z} = \dualu^2,\quad j^2 = 0.
\ela
To this end, one associates a line with a dual number by setting
\bela{E81}
  \dualu = \tan\frac{\theta}{2},\quad \dualv = s\tan\frac{\theta}{2},
\ela
where $\theta$ constitutes the oriented angle between the $x$-axis and the oriented (unit) tangent $(\tilde{v},\tilde{w})$ to the line and $s$ denotes the signed distance between the origin and the intersection of the line with the $x$-axis as indicated in Figure \ref{dual}. 
\begin{figure}
\centerline{\includegraphics[scale=0.5]{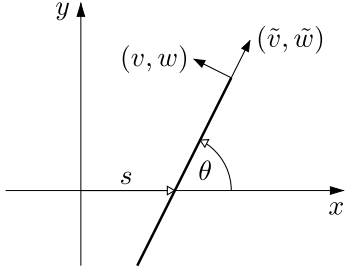}}
\caption{The geometry of a dual number $\dis z = \tan\frac{\theta}{2}\, (1 + js)$.}
\label{dual}
\end{figure}
Here, we make the assumption that we can always choose the coordinate system in such a manner that no line occurring in our discussion is ``horizontal''. This corresponds to excluding dual numbers of vanishing modulus, that is, purely ``imaginary'' dual numbers. The components of the unit tangent and the distance $s$ read
\bela{E82}
  \tilde{v} = \frac{1-\dualu^2}{1+\dualu^2},\quad \tilde{w} = \frac{2\dualu}{1+\dualu^2},\quad s = \frac{\dualv}{\dualu}
\ela
and the oriented unit normal $(v,w)$ to the line and the signed distance $d$ to the line are given by
\bela{E83}
  (v,w) = (-\tilde{w},\tilde{v}),\quad d = -s\tilde{w}.
\ela

We now observe that four lines represented by 
\bela{E83a}
 \Vsf_{(k)} = (v_{(k)},w_{(k)},0,2d_{(k)},1),\quad k=1,\ldots,4 
\ela
are in oriented contact with a circle represented by $\Vsf_\ocircle$ if the overdetermined linear system
\bela{E84}
 \left(\bear{ccc}v_{(k)}& w_{(k)}&\dis -1\ear\right)\left(\bear{c}x\\y\\ r\ear\right) = d_{(k)}
\ela
admits a solution. In terms of the real and imaginary parts of dual numbers $z_{(k)}$ representing four (non-horizontal) lines, the condition for the existence of a (unique) circle may therefore be expressed as
\bela{E84a}
   \left|\bear{cccc} 1 & \dualu_{(\cdot)} & \dualv_{(\cdot)} & \dualu_{(\cdot)}^2\ear\right| = 0
\ela
which, once again, leads to 
\bela{E85}
   \left|\bear{cccc} 1 & z_{(\cdot)} & \bar{z}_{(\cdot)} & |z_{(\cdot)}|^2\ear\right| = 0.
\ela
Hence, by virtue of the identity (\ref{E66}), we have retrieved the fact \cite{Yaglom68} that four lines represented by dual numbers $z_{(k)}$ are in oriented contact with a circle if and only if their cross-ratio (\ref{E74}) is ``real''. It is noted that for such a circle to exist, no three lines can be parallel so that, modulo a re-ordering of the lines, the cross-ratio (\ref{E74}) is well defined since $|z_{(2)}-z_{(3)}|^2\neq 0$ and $|z_{(4)}-z_{(1)}|^2\neq0$.

\subsubsection{Anti-fundamental line-circle complexes}

In analogy with the previous two cases, we now consider maps of the type
\bela{86}
 \bear{rl}
  \Lsf : \De_3 & \rightarrow \{\mbox{oriented lines in $\R^2$}\}\as
  \csf : \Do_3 & \rightarrow \{\mbox{oriented circles in $\R^2$}\}
 \ear
\ela
which we combine to maps $(\Lsf,\csf)$ defined on $\Z^3 = \De_3\cup\Do_3$. It turns out that the Laguerre analogues of fundamental point-circle and geodesic-circle complexes are not fundamental but ``anti-fundamental'' in a sense to be discussed below.

\begin{definition}
A configuration of oriented lines and circles combinatorially attached to the vertices of the even and odd sublattices of a $\Z^3$ lattice respectively is termed an {\em anti-fundamental line-circle complex} if the lines and circles are incident along the edges of the $\Z^3$ lattice.
\end{definition}

The above terminology is a reflection of the nature of the relation between four lines and four circles which are attached to the vertices of an elementary cube of a $\Z^3$ lattice and are incident along edges. Thus, if we choose three circles $\csf_1,\csf_2,\csf_3$ which are in oriented contact with a given line $\Lsf$ then, generically, there exist three unique lines $\Lsf_{lm}\neq\Lsf$, $l\neq m$ which touch the respective pairs of circles $(\csf_l,\csf_m)$ in an oriented manner. The fourth circle $\csf_{123}$ which is in oriented contact with the lines $\Lsf_{12},\Lsf_{23},\Lsf_{13}$ is then uniquely determined. On the other hand, if we regard these eight lines and circles as Lie circles then there exist two eighth Lie circles which are in oriented contact with the Lie circles $\Lsf_{12},\Lsf_{23},\Lsf_{13}$. One of them is given by $\csf_{123}$ as defined above and the other one is the ``point at infinity'' $\csf_{\infty}$ represented by $\Vsf_\infty$. It turns out that it is $\csf_\infty$ which is correlated to the Lie circle $\Lsf$. Thus, $\Lsf$ and $\csf_{123}$ are anti-correlated. In order to make good this assertion, we temporarily return to the parametrisation
\bela{E87}
 \bear{c}
 \Vsf = (M^{\emptyset,\emptyset},M^{45,45},M^{4,4},M^{5,5},M^{4,5})\as
M^{\emptyset,\emptyset}=1,\quad M^{i,k} = M^{ik},\quad M^{45,45} = \left|\bear{cc}M^{44}&M^{45}\\ M^{54}&M^{55}\ear\right|
\ear
\ela
of the points of the Lie quadric $Q^3$ with $M^{54}=M^{45}$ and the metric (\ref{E21a}). It is recalled \cite{BobenkoSchief15} that the $M$-system (\ref{E12}) governs the evolution of the minors 
\bela{E88}
  M^{A,B} = \det(M^{a_\alpha b_\beta})_{\alpha,\beta=1,\ldots,\sigma}
\ela
of the matrix $\mathcal{M}$ according to
\bela{E89}
  M^{A,B}_C = \frac{M^{CA,CB}}{M^{C,C}},
\ela
where the entries of each multi-index $A = (a_1\cdots a_\sigma)$, $B = (b_1\cdots b_\sigma)$ and $C = (c_1\cdots c_{\tilde{\sigma}})$ are assumed to be distinct elements of $\{1,2,3,4,5\}$ and $\{1,2,3\}$ respectively. Furthermore, for this to make sense, it is required that $C$ and $A\cup B$ be disjoint and we conclude that, for $C=(123)$,
\bela{E90}
  \Vsf_{123} \sim (M^{123,123},M^{12345,12345},M^{1234,1234},M^{1235,1235},M^{1234,1235}).
\ela
In order to make contact with Laguerre geometry, we choose the light-like vector $\Esf = (1,0,0,0,0)$ so that lines are represented by points $\Vsf\in Q^3$ for which $M^{45,45}= \det\hat{\mathcal{M}}$ vanishes but $\hat{\mathcal{M}}\neq 0$. The case $\hat{\mathcal{M}}=0$ corresponds to the point at infinity represented by $\Vsf_\infty = \Esf$. Hence, the four lines $\Lsf,\Lsf_{12},\Lsf_{23},\Lsf_{13}$ interpreted as Lie circles give rise to the vanishing principal minors
\bela{E91}
  M^{45,45} = M^{1245,1245} = M^{2345,2345} = M^{1345,1345} = 0.
\ela
On the other hand, since the Lie circles $\csf_1,\csf_2$ and $\csf_3$ are not lines in the Laguerre geometric sense, we deduce that $\det\hat{\calM}_l\neq0$, $l=1,2,3$ so that
\bela{E91a}
  M^{145,145}M^{245,245}M^{345,345}\neq 0.
\ela
Now, the identity
\bela{E91b}
 \bear{rcl}
  &&M^{35,45} M^{1245,1245} -  M^{25,45} M^{1245,1345}\as &+& M^{15,45} M^{1245,2345} - M^{45,45} M^{1235,1245} = 0,
 \ear
\ela
which may be directly verified for symmetric matrices $\calM$, and its two analogues obtained via permutation of the indices $1,2,3$ reveal that 
\bela{E91c}
  \left(\bear{ccc} 0&\phantom{-}M^{35,45}&-M^{25,45}\\
                        M^{35,45} & 0 & \phantom{-}M^{15,45}\\
                       M^{25,45} & \,\,-M^{15,45}\,\, &  0\ear\right)
  \left(\bear{c} M^{1245,1345}\\ M^{1245,2345}\\ M^{1345,2345}\ear\right) = 0.
 \ela
The determinant $2M^{15,45}M^{25,45}M^{35,45}$ associated with the above linear system is non-vanishing since the minors $M^{15,45}, M^{25,45}$ and $M^{35,45}$ cannot be zero by virtue of (\ref{E91a}) and $M^{45,45}=0$. Hence, we conclude that
\bela{E91d}
  M^{1245,1345} = M^{1245,2345} = M^{1345,2345} = 0
\ela
so that all minors of the form $M^{\ast\ast 45,\ast\ast 45}$ vanish. This implies, in turn, that the symmetric matrix $\mathcal{M}$ has rank 3 due to the non-degeneracy condition (\ref{E91a}). In particular, all minors in $\Vsf_{123}$ except for $M^{123,123}$ vanish or, equivalently, \mbox{$\hat{\calM}_{123}=0$}. The latter confirms that $\Vsf$ and $\Vsf_{123}=\Vsf_\infty$ are correlated. This fact translates into a geometric theorem which states that if the circles $\csf$ and $\csf_{12},\csf_{23},\csf_{13}$ in Figure \ref{eight_circle} are replaced by lines then, remarkably, the circles $\csf^{12}_3,\csf^{23}_1,\csf^{13}_2$ turn out to be lines as indicated in Figure \ref{laguerre_infinity}.
\begin{figure}
  \centerline{\includegraphics[scale=0.35]{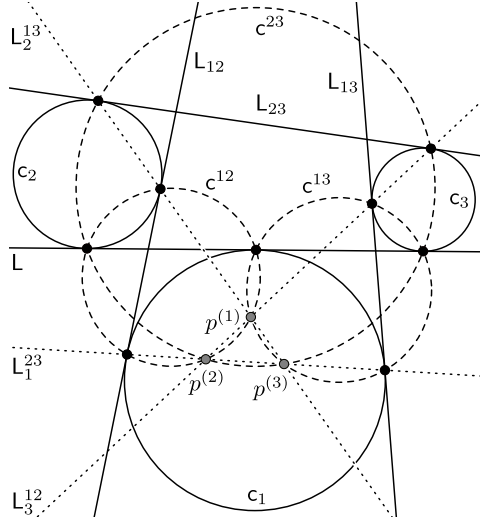}}
  \caption{The analogue of Figure \ref{eight_circle} wherein the circles $\csf$ and $\csf_{12},\csf_{23},\csf_{13}$ are replaced by lines $\Lsf$ and $\Lsf_{12},\Lsf_{23},\Lsf_{13}$. As a result, the circles $\csf^{12}_3,\csf^{23}_1,\csf^{13}_2$ become lines $\Lsf^{12}_3,\Lsf^{23}_1,\Lsf^{13}_2$.}
\label{laguerre_infinity}
\end{figure}

The construction and geometric integrability of anti-fundamental line-circle complexes may be dealt with in the same manner as in the case of fundamental point-circle and geodesic-circle complexes, leading to the following theorem.

\begin{theorem}
Anti-fundamental line-circle complexes $(\Lsf,\csf)$ are uniquely determined by one-dimensional Cauchy data of the combinatorial type depicted in Figure \ref{star} (left). These may be extended to lattices of $\Z^4$ combinatorics by virtue of the canonical Laguerre geometric analogue of Clifford's circle theorem.
\end{theorem}

In algebraic terms, anti-fundamental line-circle complexes are governed by a suitable re-interpretation of the nine-point equation
\bela{E92}
  \frac{\Msf_{(1)}}{\bar{\Msf}^*_{(1)}} + \frac{\Msf_{(2)}}{\bar{\Msf}^*_{(2)}}  = 1.
\ela
To this end, we first note that the defining relation $j^2=0$ of dual numbers gives rise to the identity
\bela{E93}
  \Re\left[\frac{\Msf_{(1)}}{\bar{\Msf}^*_{(1)}} + \frac{\Msf_{(2)}}{\bar{\Msf}^*_{(2)}}\right] = 
  \Re\left[\frac{\Msf_{(1)}}{\Msf^*_{(1)}} + \frac{\Msf_{(2)}}{\Msf^*_{(2)}}\right]
\ela
so that, by virtue of (\ref{E92}), the real part of the nine-point equation for dual numbers vanishes identically. Accordingly, the latter is not sufficient to determine $c^0$ in terms of the other eight dual numbers. Nevertheless, if we consider the generalised complex numbers \cite{Yaglom68,HarkinHarkin04}
\bela{E96}
 z = \dualu + j \dualv,\quad \bar{z} = \dualu - j \dualv,\quad |z|^2 = z\bar{z} = \dualu^2 + \epsilon \dualv^2,  \quad j^2 = \epsilon,
 \ela
where $\epsilon$ is an arbitrary but fixed real number, then the same argument shows that, on removing the denominator, the nine-point equation (\ref{E92}) for generalised complex numbers is of the form
\bela{E97}
  j p(j)= 0,
\ela
where $p$ is a ``real'' polynomial in $j$. If we now make use of the relation $j^2=\epsilon$ then this assumes the form
\bela{E98}
  j[f(\epsilon) + jg(\epsilon)] = 0.
\ela
Hence, for $\epsilon\neq0$, this is equivalent to the two equations $f=g=0$ which uniquely determine $c^0$. If $\epsilon=0$ then (\ref{E98}) may not be solved for $f+ jg$ and we only obtain $f=0$ as a necessary condition. However, it turns out that we can regard the case $\epsilon=0$ as a limiting case in which the condition $g=0$ is still valid so that anti-fundamental line-circle complexes are governed by $p(j) = 0$ or, in short-hand notation,
\bela{E100}
  \left.\left[\frac{1}{j}\left(\frac{\Msf_{(1)}}{\bar{\Msf}^*_{(1)}} + \frac{\Msf_{(2)}}{\bar{\Msf}^*_{(2)}}  - 1\right)\right]\right|_{\epsilon=0} = 0.
\ela

\section{Perspectives}

In this paper, we have discussed the geometry of canonical reductions of the integrable $M$-system (cf.\ Figure \ref{summary}). We have analysed in detail real reductions of the $M$-system, while the case of Hermitian $\mathcal{M}$ has been consigned to a future publication. We have briefly indicated how our analysis leads to an integrable nine-point lattice equation over the complex, dual and double numbers. It is therefore natural to consider the $M$-system itself over algebras of hypercomplex numbers \cite{KantorSolodovnikov89} such as the associative-commutative algebras (\ref{E96}). Up to isomorphisms, the latter consist of the complex, dual and double numbers corresponding to $\epsilon=-1,0,1$ respectively. It is evident that the algebraic integrability of the $M$-system is not affected by the choice of $\epsilon$ and, formally, line geometry over these hypercomplex numbers may be defined in the standard manner. In this way, we obtain fundamental line complexes over complex, dual and double numbers. Moreover, the Hermitian reduction $\calM^\dagger = \calM$ remains admissible. In the case of complex numbers, this gives rise to the above-mentioned interpretation of the corresponding fundamental line complexes as fundamental sphere complexes in Lie geometry. In general, the set of points in the ``hypercomplex'' Pl\"ucker quadric associated with the Hermitian reduction are given by
\bela{E102}
  \Vsf = \left(1,\left|\bear{cc}M^{44}&M^{45}\\ \bar{M}^{45}&M^{55}\ear\right|,M^{44},M^{55},\bar{M}^{45},M^{45}\right).
\ela
As in the complex case, real homogeneous coordinates are obtained by considering
\bela{E103}
  \hat{\Vsf} = \left(1,\left|\bear{cc}M^{44}&M^{45}\\ \bar{M}^{45}&M^{55}\ear\right|,M^{44},M^{55},\Re( M^{45}),\Im(M^{45})\right)
\ela
and the induced metric of the space of homogeneous coordinates is given by
\bela{E104}
  \diag\left[\left(\bear{cc}0&1\\1&0\ear\right),-\left(\bear{cc}0&1\\1&0\ear\right),2,-2\epsilon\right].
\ela
Once again, the condition $\langle\hat{\Vsf},\hat{\Vsf}_\ast\rangle = 0$ is equivalent to $\langle\Vsf,\Vsf_\ast\rangle=0$. Hence, dual numbers ($\epsilon=0$) and double numbers ($\epsilon=1$) give rise to circle complexes in Lie geometry and line complexes in real Pl\"ucker geometry respectively. In fact, since the real part of the $M$-system over dual numbers is identical with the $M$-system for the real part of dual numbers and the imaginary part is effectively irrelevant in the parametrisation of $\hat{V}$, the (real) circle complexes associated with dual numbers turn out to be fundamental. Moreover, if we set
\bela{E105}
  M^{ik} = A^{ik} + j B^{ik},\quad A^{ik}= A^{ki},\quad B^{ik} = -B^{ki}
\ela
for $\epsilon=1$ then it is readily verified that the Hermitian reduction of the $M$-system is equivalent
to the real $M$-system for
\bela{E106}
  \tilde{M}^{ik} = A^{ik} + B^{ik}
\ela
and, hence, the real line complexes associated with double numbers are likewise fundamental. Accordingly, we have retrieved the standard fundamental sphere, circle and real line complexes by considering the same type of reduction of the $M$-system over different hypercomplex numbers. Finally, the same approach applied to quaternions, which constitute particular non-commutative hypercomplex numbers, leads to fundamental complexes of four-dimensional hyperspheres in Lie geometry. This is the subject of a forthcoming publication.

\section*{Acknowledgements}

This research was supported by the DFG Collaborative Research Centre\linebreak SFB/TRR 109 {\em Discretization in Geometry and Dynamics} and the Australian Research Council (ARC). One of the authors (W.K.S.) would like to thank Alastair King for enlightening discussions during the workshop  {\em Discrete Differential Geometry}, Mathematisches Forschungsinstitut Oberwolfach, March 1st--7th, 2015 which resulted in Theorems \ref{alastair1} and \ref{alastair2}. We also wish to express our gratitude to Tim Hoffmann whose initial conjecture led to Theorem \ref{tim} and to Boris Springborn for numerous insightful comments on the subject matter.


\begin{thebibliography}{99}

\bibitem{Finikov59} Finikov, SP. 1959 
{Theorie der Kongruenzen}.
Akademie-Verlag: Berlin.

\bibitem{Eisenhart60} Eisenhart, LP. 1960
{A treatise on the differential geometry of curves and surfaces}. 
New York: Dover.

\bibitem{Doliwa01} Doliwa, A. 2001
Discrete asymptotic nets and W-congruences in Pl\"ucker line geometry.
{\sl J.\ Geom.\ Phys.\ }{\bf 39}, 9-29.

\bibitem{BobenkoSuris09} Bobenko, AI, Suris, YB. 2008 
{\sl Discrete differential geometry. Integrable structure}.
Graduate Studies in Mathematics {\bf 98}. Providence: AMS.

\bibitem{DoliwaSantiniManas00} Doliwa, A, Santini, PM, Manas, M. 2000
Transformations of quadrilateral lattices.
{\sl J.\ Math.\ Phys.\ }{\bf 41}, 944-990.

\bibitem{Sauer50} Sauer, R. 1950
Parallelogrammgitter als Modelle f\"ur pseudosph\"arische Fl\"achen.
{\sl Math.\ Z.\ }{\bf 52}, 611-622.

\bibitem{Wunderlich51} Wunderlich, W. 1951
Zur Differenzengeometrie der Fl\"achen konstanter negativer Kr\"ummung.
{\sl \"Osterreich.\ Akad.\ Wiss.\ math.-nat.\ Kl.\ S.-B.\ II} {\bf 160}, 39-77.

\bibitem{BobenkoPinkall96} Bobenko, A, Pinkall, U. 1996
Discrete surfaces with constant negative Gaussian curvature and the Hirota equation.
{\sl J.\ Diff.\ Geom.\ } {\bf 43}, 527-611.

\bibitem{Schief03} Schief, WK. 2003 
On the unification of classical and novel integrable surfaces. II.\ Difference geometry. 
{\sl Proc.\ R.\ Soc.\ London A} {\bf 459}, 373-391.

\bibitem{Schief06} Schief, WK. 2006
On a maximum principle for minimal surfaces and their integrable discrete counterparts.
{\sl J.\ Geom.\ Phys.\ }{\bf 56}, 1484-1495.

\bibitem{BobenkoPottmannWallner2010} Bobenko, AI, Pottmann, H, Wallner, J. 2010
A curvature theory for discrete surfaces based on mesh parallelity.
{\sl Math.\ Annalen} {\bf 348}, 1-24.

\bibitem{WangJiangBompasPottmann13} Wang, J, Jiang, C, Bompas, P, Wallner, J, Pottmann, H. 2013
Discrete line congruences for shading and lighting.
{\sl SGP '13 Proceedings of the Eleventh Eurographics/ACMSIGGRAPH Symposium on Geometry Processing}, pp. 53-62.
The Eurographics Association and Blackwell Publishing Ltd.

\bibitem{BobenkoSchief15} Bobenko, AI, Schief, WK. 2015
Discrete line complexes and integrable evolution of minors.
{\sl Proc.\ R.\ Soc.\ London A} {\bf 471}, 20140819 (23pp).

\bibitem{OnishchikSulanke06} Onishchik, AL,  Sulanke, R. 2006
{\sl Projective and Cayley-Klein geometries}.
Springer Monographs in Mathematics. Berlin: Springer-Verlag.

\bibitem{SempleKneebone52} Semple, JG, Kneebone, GT. 1952
{\sl Algebraic projective geometry}.
Oxford: Oxford University Press.

\bibitem{Blaschke29} Blaschke, W. 1929
{\sl Vorlesungen \"uber Differentialgeometrie. III.\ Differentialgeometrie der Kreise und Kugeln}.
Berlin: Springer-Verlag.

\bibitem{Kashaev96} Kashaev, RM. 1996
On discrete three-dimensional equations associated with the local Yang-Baxter relation.
{\sl Lett.\ Math.\ Phys.\ }{\bf 33}, 389-397.

\bibitem{Schief03b} Schief, WK. 2003
Lattice geometry of the discrete Darboux, KP, BKP and CKP equations. Menelaus' and
Carnot's theorems.
{\sl J.\ Nonlinear Math.\ Phys.\ }{\bf 10}, Supplement 2, 194-208.

\bibitem{KenyonPemantle13} Kenyon, R, Pemantle, R. 2013
Double-dimers, the Ising model and the hexahedron recurrence.
arXiv:1308.2998 [math-ph]. 

\bibitem{HoltzSturmfels07} Holtz, O, Sturmfels, B. 2007
Hyperdeterminantal relations among symmetric principal minors.
{\sl J.\ Algebra} {\bf 316}, 634-648.

 \bibitem{Coolidge16} Coolidge, JL. 1916 
{\sl A Treatise on the circle and the sphere}. 
Oxford: Clarendon Press.

\bibitem{Clifford71} Clifford, WK. 1871 
A synthetic proof of Miquel's theorem.
{\sl Oxford, Cambridge and Dublin Messenger of Math.\ }{\bf 5}, 124-141.

\bibitem{Pedoe88} Pedoe, D. 1988
{\sl Geometry: A comprehensive course}.
New York: Dover.

\bibitem{Yaglom68} Yaglom, IM. 1968
{\sl Complex numbers in geometry}.
New York: Academic Press.

\bibitem{BobenkoHoffmannSuris02} Bobenko, AI, Hoffmann, T, Suris, Yu.B. 2002
Hexagonal circle patterns and integrable systems. Patterns with the multi-ratio property and Lax equations on the regular triangular lattice.
{\sl International Math.\ Research Notices} {\bf 2002}, 111-164.

\bibitem{BobenkoHoffmann03} Bobenko, AI, Hoffmann, T. 2003
Hexagonal circle patterns and integrable systems. Patterns with constant angles.
{\sl Duke Math.\ J.\ }{\bf 116}, 525-566.

\bibitem{Plucker65} Pl\"ucker, J. 1865
On a new geometry of space.
{\sl Phil.\ Trans.\ R.\ Soc.\ London} {\bf 155}, 725-791.

\bibitem{Carver05} Carver, WB. 1905
On the Cayley-Veronese class of configurations.
{\sl Trans.\ Amer.\ Math.\ Soc.\ }{\bf  6}, 534-545.

\bibitem{GelfandKapranovZelevinsky94} Gelfand, IM, Kapranov, MM, Zelevinsky, AV. 1994
{\sl Discriminants, resultants, and multidimensional determinants}.
Boston: Birkh\"auser.

\bibitem{TsarevWolf09} Tsarev, SP,  Wolf, T. 2009
Hyperdeterminants as integrable discrete systems. 
{\sl J.\ Phys.\ A: Math.\ Theor.\ }{\bf 42}, 454023 (9pp).

\bibitem{RogersSchief02} Rogers, C, Schief, WK. 2002 
{\sl Darboux and B\"acklund transformations. Geometry and modern applications in soliton theory}. 
Cambridge Texts in Applied Mathematics. Cambridge, UK: Cambridge University Press (2002).

\bibitem{Strubecker53} Strubecker, K. 1953
Kinematik, Liesche Kreisgeometrie und Geraden-Kugel-Transformation.
{\sl Elemente der Mathematik} {\bf 8}, 4-13.

\bibitem{LonguetHiggins72}  Longuet-Higgins, MS. 1972 
Clifford's chain and its analogues in relation to the higher polytopes. 
{\sl Proc.\ R.\ Soc.\ London A} {\bf 330}, 443-466.

\bibitem{Hyperbolic} Coxeter, HSM. 1998
{\sl Non-Euclidean geometry}. 6th Edition.
Washington, DC: Mathematical Association of America.

\bibitem{HarkinHarkin04} Harkin, AA, Harkin, JB. 2004
Geometry of generalized complex numbers.
{\sl Mathematics Magazine} {\bf 77}, 118-129.

\bibitem{KantorSolodovnikov89} Kantor, IL, Solodovnikov, AS. 1989 
{\sl Hypercomplex numbers}.
Berlin, New York: Springer-Verlag.

\end{thebibliography}
\end{document}